\documentclass[12pt, a4paper]{scrartcl}

\usepackage[T1]{fontenc}
\usepackage{lmodern}
\usepackage[english]{babel}
\usepackage{amsfonts}
\usepackage{amsmath}
\usepackage{amssymb}
\usepackage{amsthm}
\usepackage{mathrsfs}
\usepackage[dvipsnames,svgnames]{xcolor}
\usepackage{booktabs}
\usepackage{float}
\usepackage{enumitem}
\usepackage{graphicx}
\usepackage{tikz,pgf}
\usetikzlibrary{calc}
\usetikzlibrary{decorations.markings}
\usetikzlibrary{shapes.geometric}
\usetikzlibrary{patterns}
\usetikzlibrary{intersections}
\usetikzlibrary{arrows.meta}
\usetikzlibrary{fadings}
\usepackage{multirow}
\usepackage{array}
\usepackage{adjustbox}
\usepackage{calc}
\usepackage[colorlinks=true,linkcolor=RoyalBlue,urlcolor=RoyalBlue,citecolor=PineGreen]{hyperref}
\usepackage[nameinlink]{cleveref}

\theoremstyle{plain}
\newtheorem{lemma}{Lemma}[section]
\newtheorem{proposition}[lemma]{Proposition}
\newtheorem{theorem}[lemma]{Theorem}
\newtheorem*{theoremnonr}{Theorem}
\newtheorem{corollary}[lemma]{Corollary}

\theoremstyle{definition}
\newtheorem{definition}[lemma]{Definition}

\newtheorem{remark}[lemma]{Remark}

\newcommand{\N}{\mathbb{N}}
\newcommand{\Z}{\mathbb{Z}}

\newcommand{\R}{\mathbb{R}}

\newcommand{\Bipyramid}{{\mathbb{B}_n}}
\newcommand{\Penta}{{\mathbb{B}_5}}
\newcommand{\Top}{\mathsf{T}}
\newcommand{\Bottom}{\mathsf{B}}
\newcommand{\Pyramid}[1]{\mathbb{P}_{#1}}
\newcommand{\Tetrahedron}[2]{\mathbb{T}_{{#1},{#2}}}
\newcommand{\OCM}{\mathrm{CM}^{\ast}_n}
\newcommand{\Eq}[1][n]{{\mathcal{E}_{#1}}}
\newcommand{\motion}{\mathcal{M}}
\newcommand{\Gal}{\mathscr{G}}
\newcommand{\ones}[1]{{#1}^{(\mathbf{1})}}
\newcommand{\distance}{d_{\Top\Bottom}}
\newcommand{\Oct}{\mathbb{O}}
\newcommand{\kasperl}{\mathsf{k}}
\newcommand{\pezi}{\mathsf{p}}
\newcommand{\Subdivided}{\mathbb{S}}
\newcommand{\New}{\mathsf{N}}

\colorlet{colbg}{white}
\colorlet{colfg}{black}
\colorlet{colG}{DarkSeaGreen}
\definecolor{colR}{HTML}{CC6677}
\definecolor{colO}{HTML}{DDCC77}
\definecolor{colB}{HTML}{6699CC}
\colorlet{vcol}{colfg!75!white}
\colorlet{ecol}{colfg!55!white}

\tikzstyle{vertex}=[fill=vcol,circle,inner sep=0pt, minimum size=4pt]
\tikzstyle{edge}=[line width=1.5pt,ecol]
\tikzstyle{labelsty}=[font=\scriptsize]
\tikzstyle{cdiag}=[line width=0.5pt,-{Classical TikZ Rightarrow[]}]
\tikzstyle{cdiagd}=[cdiag,dashed]

\tikzstyle{hedge}=[edge,colB]
\tikzstyle{oedge}=[edge,ecol!50!white]

\title{Pentagonal bipyramids lead to the smallest flexible embedded polyhedron}
\date{}

\author{%
Matteo Gallet$^{\diamond}$%
\and
Georg Grasegger$^{\ast}$%
\and
Jan Legersk\'y$^{\circ}$%
\and
Josef Schicho%
}
\renewcommand{\thefootnote}{\fnsymbol{footnote}}

\begin{document}
\maketitle
\footnotetext{\hspace{0.15cm}$^\circ$ Supported by the Czech Science Foundation (GAČR), project No. 22-04381L.\\%
$^\diamond$ The researcher is a member of ``Gruppo Nazionale per le Strutture Algebriche, Geometriche e le loro Applicazioni'', INdAM.\\%
$^\ast$ Supported by the Austrian Science Fund (FWF): 10.55776/I6233.\\\\%
This research was funded in whole, or in part, by the Austrian Science Fund (FWF) 10.55776/I6233. For the purpose of open access, the authors have applied a CC BY public copyright license to any Author Accepted Manuscript version arising from this submission.
}
\begin{abstract}
 Steffen's polyhedron
 was believed to have the least number of vertices among
 polyhedra that can flex without self-intersections.
 Maksimov clarified that the pentagonal bipyramid
 with one face subdivided into three is the only polyhedron with fewer vertices for which the existence of a self-intersection-free flex was open.
 Since subdividing a face into three does not change the mobility, we focus on flexible pentagonal bipyramids.
 When a bipyramid flexes,
 the distance between the two opposite vertices of the two pyramids changes;
 associating the position of the bipyramid to this distance
 yields an algebraic map that determines a nontrivial extension
 of rational function fields.
 We classify flexible pentagonal bipyramids
 with respect to the Galois group of this field extension
 and provide examples for each class, building on a construction proposed by Nelson.
 Surprisingly, one of our constructions yields a flexible pentagonal bipyramid
 that can be extended to an embedded flexible polyhedron with 8 vertices.
 The latter hence solves the open question.
\end{abstract}
\renewcommand{\thefootnote}{\arabic{footnote}}

\section{Introduction}
\label{introduction}

Eighty years after Bricard's description of three families
of self-intersecting flexible octahedra (see~\cite{Bricard1897}),
Connelly constructed an example of a flexible embedded polyhedron,
namely one in which any two faces intersect only at their common edge (see \cite{Connelly1977}).
Connelly's construction yields a polyhedron with $18$ vertices (see~\cite{Connelly1979} and~\cite[Remark~4]{Kuiper1979})
but later Steffen \cite{Steffen1978} came up with a construction
of a flexible embedded polyhedron with only $9$ vertices,
which is therefore now known as \emph{Steffen's polyhedron}.
Very recently, an example
of a flexible embedded polyhedron with $26$ vertices,
all of whose dihedral angles change during a flex, was constructed \cite{Alexandrov2024}.
Maksimov in \cite{Maksimov1995} states that there are no flexible embedded polyhedra
with less than $9$ vertices; this result has been referenced in~\cite{Cromwell1997} and~\cite{Demaine2002}.
However, in~\cite{Maksimov2008} Maksimov points out that there is a single case,
among polyhedra with $8$ vertices, for which the non-existence of flexible embedded instances is still not clear.
Both Gaifullin in~\cite{Gaifullin2018} and Alexandrov in~\cite{Alexandrov2020} agree
that the question whether $9$ is the minimal number of vertices for embedded flexibility is still open.

The missing case,
for which the existence of a flexible embedded polyhedron has not been proved or disproved yet,
is that of a polyhedron obtained from a pentagonal bipyramid
one of whose faces has been subdivided into three;
in other words, we remove a face, add a new vertex, and add three new faces determined by the new vertex and the three edges of the removed face.
Here, by an \emph{$n$-gonal bipyramid} we mean a polyhedron whose combinatorial structure is given by an $n$-cycle all of whose vertices are connected to two further ones, which are not neighbors. These objects are also called \emph{suspensions} and the cycle is called \emph{equator}; see \cite{Alexandrov2011, Connelly1978, Maksimov1994}. We usually refer to $4$-gonal bipyramids as to \emph{octahedra}.

The mobility does not essentially change by subdividing a face,
that is why we focused our attention on flexible pentagonal bipyramids.
We tried to adopt the technique developed in~\cite{Gallet2020}
(which uses the tools introduced in \cite{Gallet2019})
to classify all the possible non-degenerate motions\footnote{Here by ``degenerate''
we mean a motion during which two faces stay always coplanar or, equivalently, a dihedral angle at an edge is frozen.} of such a polyhedron,
but unfortunately this led only to partial results.
Therefore, we decided to pursue a coarser classification,
based on an invariant which we call the \emph{Galois group of a motion}.

In a non-degenerate motion of a pentagonal bipyramid
the distance between the two non-equatorial vertices is not constant.
Therefore, the map that associates such distance to a configuration is not constant.
This means that, once we can think of the set of configurations of the bipyramid
during the motion as of an algebraic variety, what we get is a rational function on the configuration space.
In turn, this determines a field extension between the field of rational functions on the affine line
(where the distance takes its values) and the field of rational functions on the configuration space.
As for every field extension, we can define its Galois group,
which is then the Galois group of the given motion.

The first main result of this paper is the following
(see \Cref{proposition:galois_groups} and \Cref{theorem:nelson_Z2}, \Cref{theorem:nelson_opposite}):

\begin{theoremnonr}
 There are only two possibilities for the Galois group of a non-degenerate motion
 of a pentagonal bipyramid.
 Both these possibilities arise as Galois groups of some particular flexible pentagonal bipyramids
 obtained via \emph{Nelson constructions}.
\end{theoremnonr}

A \emph{Nelson construction} is a way to obtain a flexible pentagonal bipyramid
by ``gluing'' two flexible Bricard octahedra,
following the procedure described by Nelson in~\cite{Nelson2010} (see also~\cite{Nelson2012}).

We realized that the Nelson construction that we have employed to exhibit a witness for one of the two possible Galois groups, namely, the one involving gluing two Bricard octahedra of so-called Type~I, or \emph{line-symmetric}, could be suitable for finding an example of a flexible embedded polyhedron with $8$ vertices.
Indeed, aided by a computer exploration, we could reach the second main result of this paper
(see \Cref{fig:embedded} and \Cref{theorem:existence_flexible}):

\begin{theoremnonr}
 There exists a flexible embedded polyhedron with $8$ vertices,
 obtained by gluing two line-symmetric octahedra via a Nelson construction
 and then subdividing one of the faces into three by introducing a new vertex.
\end{theoremnonr}

\begin{figure}[H]
	\centering
	\includegraphics[width=7cm,clip=true,trim= 0.1cm 0.6cm 0.1cm 0.9cm]{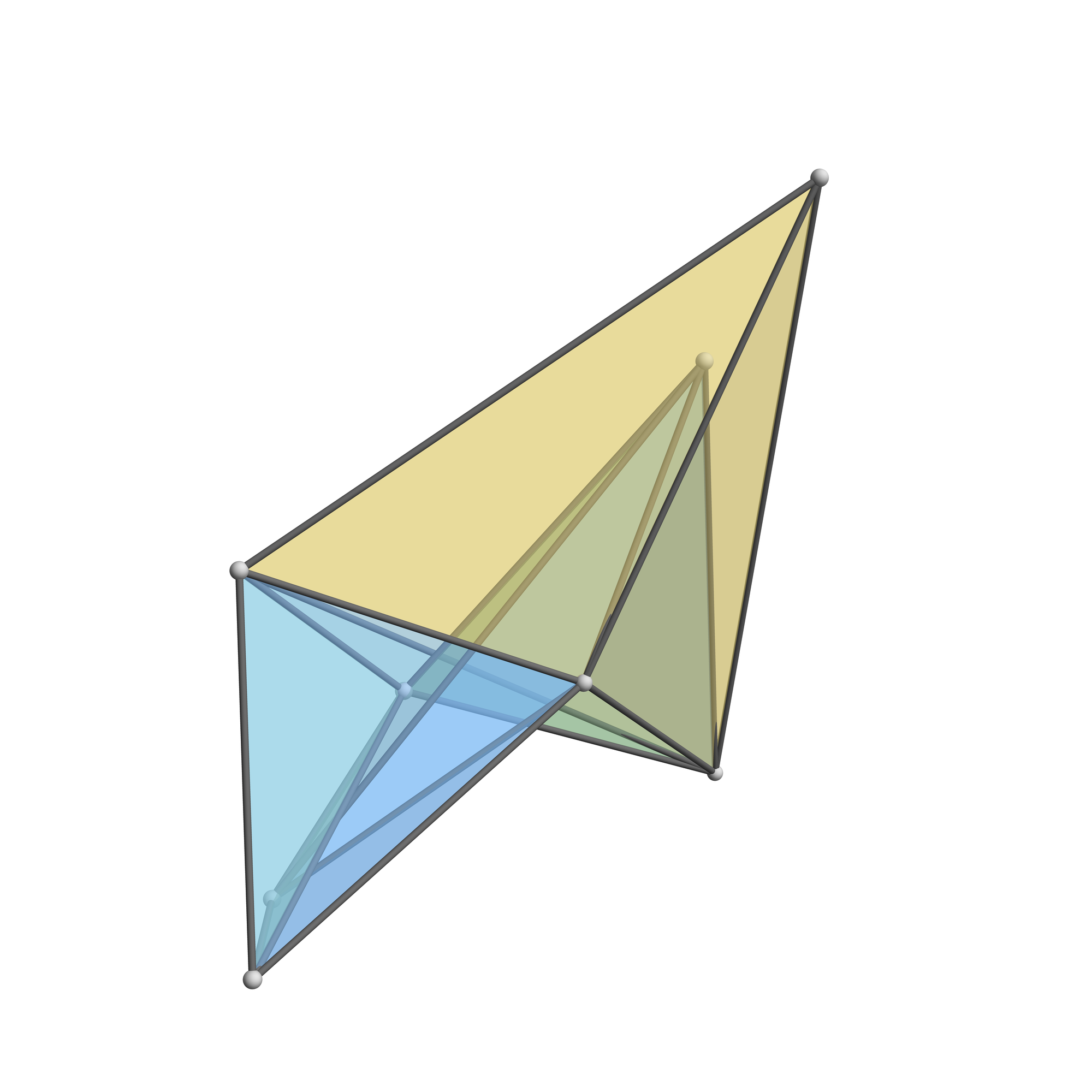}
	\caption{An example of a flexible embedded polyhedron with 8 vertices.}
	\label{fig:embedded}
\end{figure}

Due to the previously mentioned results in the literature, this flexible embedded polyhedron has the minimal possible number of vertices.

The paper is structured as follows.
\Cref{basics} introduces the basic formalism to discuss about realizations and motions of pentagonal bipyramids,
and recalls Nelson constructions.
\Cref{galois} sets up the Galois theory for motions of bipyramids, shows that there are only two possible Galois groups,
and provides an example for each of them via Nelson constructions.
\Cref{example} exhibits a flexible embedded polyhedron with $8$ vertices.

In case the reader is mainly interested in the example of a flexible embedded polyhedron with $8$ vertices,
most of the material in this paper can be skipped.
A roadmap to the example is given by: \Cref{basics:graphs},
\Cref{definition:realization,definition:flexible}, \Cref{basics:octahedra},
the Nelson construction in the proof of \Cref{theorem:nelson_opposite}, and then \Cref{example}.

\section{Basic definitions}
\label{basics}

In this section, we provide some introductory notions about Galois theory and
we define the subgraphs of the skeleton of a bipyramid that are relevant for our arguments.
We also introduce the notions describing flexibility.
We conclude the section with recalling the types of Bricard octahedra and Nelson construction.

\subsection{Basics of Galois theory}

Galois theory is a subfield of algebra and number theory that was originally developed to study the solvability of univariate polynomial equations by investigating a group, called the \emph{Galois group} of the given polynomial. The idea is, given a polynomial $f \in K[x]$, where $K$ is a field, to construct a smallest superfield $F$ of $K$ where $f$ factors into linear polynomials (such a field is called a \emph{splitting field} of $f$ over $K$); then, the Galois group of $f$ is the group of automorphisms of $F$ that fix $K$. This group can be thought as a subgroup of permutations that swap the roots of~$f$: the structure of $f$ determines which permutations are allowed and which are not.

We start by providing some basic definitions for Galois theory. Possible references are, among others, the books \cite{Artin1998, Jacobson1985, Artin1991}.

\begin{definition}
 An \emph{extension of fields} (or \emph{field extension}) is an inclusion of fields $K \subset F$; since every non-trivial homomorphism of fields is injective, we can consider non-trivial homomorphisms of fields as a field extension.
 A field extension $K \subset F$ is called \emph{Galois} if $F$ is a splitting field of a polynomial $f \in K[x]$ that factors into distinct linear polynomials in $F[x]$.
 The \emph{Galois group} of a field extension $K \subset F$ is the group
 \[
  \mathrm{Gal}(F/K) := \bigl\{ \varphi \in \mathrm{Aut}(F) \, \mid \, \varphi|_{K} = \mathrm{id}_{K} \bigr\} \,,
 \]
 namely, the group of automorphisms of~$F$ that leave $K$ fixed.
\end{definition}

In our setting, the Galois group of a flexible pentagonal bipyramid arises as an instance of the following construction.
A dominant rational map $f \colon C \dashrightarrow D$ between real algebraic curves~$C$ and~$D$
determines a field extension $\R(D) \subset \R(C)$ between the function fields of~$D$ and~$C$,
respectively.
Recall that the function field of an algebraic variety $X \subset \R^n$ is the fraction field of the coordinate ring $\R[x_1, \dotsc, x_n]/I(X)$, where $I(X)$ is the ideal of polynomials vanishing on~$X$.
Hence, we can associate to $f$ a Galois group, namely the group $\mathrm{Gal}\bigl(\R(C)/\R(D)\bigr)$.

\subsection{Relevant graphs and subgraphs}
\label{basics:graphs}

We fix the following notation for the combinatorial structure of an $n$-gonal bipyramid,
whose graph we denote by $\Bipyramid=(V_\Bipyramid,E_\Bipyramid)$; see \Cref{figure:double_penta} for the combinatorial structure of~$\Penta$.
The $n$ vertices on the equator are labeled by the numbers from $1$ to $n$;
the remaining two vertices are labeled~$\Top$ and~$\Bottom$ (for ``top'' and ``bottom'').
Throughout the paper, we always suppose $n \geq 4$, since $3$-gonal bipyramids are rigid.

For our discussion, it is useful to consider \emph{almost tetrahedra},
namely subgraphs of~$\Bipyramid$ induced by the vertices $\Top$, $\Bottom$, $i$, $i+1$
(modulo~$n$) for $i \in \{1, \dotsc, n\}$.
Notice that each of these subgraphs is isomorphic to the $1$-skeleton of a ``tetrahedron with a missing edge'',
where the missing edge is $\{\Top, \Bottom\}$.
The almost tetrahedron determined by the vertices~$i$ and~$i+1$ is denoted by~$\Tetrahedron{i}{i+1}$
(see \Cref{figure:double_penta} for illustrations).

\begin{figure}[H]
  \centering
  \begin{tikzpicture}[scale=0.9]
    \node[vertex, label={[labelsty,label distance=-2pt]90:$\Top$}] (T) at (0.89, 4.56) {};
    \node[vertex, label={[labelsty,label distance=-2pt]270:$\Bottom$}] (B) at (1.50,0.24) {};
    \node[vertex, label={[labelsty,label distance=4pt, anchor=center]240:$1$}] (1) at (0.99,1.91) {};
    \node[vertex, label={[labelsty,label distance=-2pt]180:$2$}] (2) at (0,2.25) {};
    \node[vertex, label={[labelsty,label distance=-2pt]-60:$3$}] (3) at (1.21,3.12) {};
    \node[vertex, label={[labelsty,label distance=-2pt]0:$4$}] (4) at (2.38,2.78) {};
    \node[vertex, label={[labelsty, label distance=4pt, anchor=center]150:$5$}] (5) at (1.87,2.05) {};
    \draw[edge] (T)--(1);
    \draw[edge] (B)--(1);
    \draw[edge] (T)--(2);
    \draw[edge] (B)--(2);
    \draw[edge] (T)--(3);
    \draw[edge] (B)--(3);
    \draw[edge] (T)--(4);
    \draw[edge] (B)--(4);
    \draw[edge] (T)--(5);
    \draw[edge] (B)--(5);
    \draw[edge] (1)--(2);
    \draw[edge] (2)--(3);
    \draw[edge] (3)--(4);
    \draw[edge] (4)--(5);
    \draw[edge] (5)--(1);
  \end{tikzpicture}
  \qquad
    \begin{tikzpicture}[scale=0.9]
    \node[vertex, label={[labelsty,label distance=-2pt]90:$\Top$}] (T) at (0.89, 4.56) {};
    \node[vertex, label={[labelsty,label distance=-2pt]270:$\Bottom$}] (B) at (1.50,0.24) {};
    \node[vertex, label={[labelsty,label distance=4pt, anchor=center]240:$1$}] (1) at (0.99,1.91) {};
    \node[vertex, label={[labelsty,label distance=-2pt]180:$2$}] (2) at (0,2.25) {};
    \node[vertex, label={[labelsty,label distance=-2pt]-60:$3$}] (3) at (1.21,3.12) {};
    \node[vertex, label={[labelsty,label distance=-2pt]0:$4$}] (4) at (2.38,2.78) {};
    \node[vertex, label={[labelsty, label distance=4pt, anchor=center]150:$5$}] (5) at (1.87,2.05) {};
    \draw[oedge] (T)--(3);
    \draw[oedge] (B)--(3);
    \draw[oedge] (T)--(4);
    \draw[oedge] (B)--(4);
    \draw[oedge] (T)--(5);
    \draw[oedge] (B)--(5);
    \draw[oedge] (2)--(3);
    \draw[oedge] (3)--(4);
    \draw[oedge] (4)--(5);
    \draw[oedge] (5)--(1);
    \draw[hedge] (T)--(1);
    \draw[hedge] (B)--(1);
    \draw[hedge] (T)--(2);
    \draw[hedge] (B)--(2);
    \draw[hedge] (1)--(2);
  \end{tikzpicture}
  \caption{The graph of a pentagonal bipyramid (left) and
  of the almost tetrahedron $\Tetrahedron{1}{2}$ (right).}
  \label{figure:double_penta}
\end{figure}
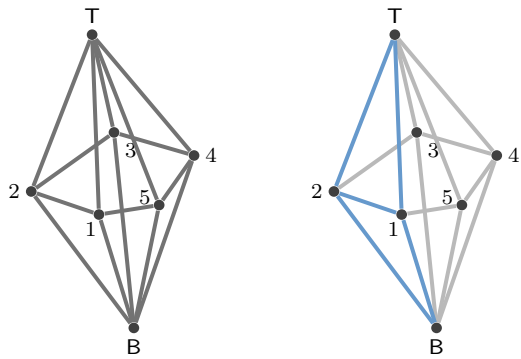

\subsection{Motions}
\label{basics:motions}
In the following we describe the notions of realizations and motions.
\begin{definition}
\label{definition:realization}
A \emph{realization} of~$\Bipyramid$ is a map
\[
 \rho \colon V_{\Bipyramid} \longrightarrow \R^3
\]
such that the images of vertices connected by an edge are distinct.
Two realizations are \emph{congruent} if one can be obtained from the other by applying a direct isometry of~$\R^3$.
\end{definition}

\begin{definition}
\label{definition:flexible}
 A realization~$\rho$ of~$\Bipyramid$ induces \emph{edge lengths} $\lambda_{uv} := \left\| \rho(u) - \rho(v) \right\|$ for each $\{u,v\} \in E_{\Bipyramid}$.
 Once we fix a tuple of edge lengths $\lambda = (\lambda_{e} \, \colon \, e \in E_{\Bipyramid})$,
 we say that the $n$-gonal bipyramid is \emph{flexible}
 if there exist infinitely many non-congruent realizations of~$\Bipyramid$ inducing~$\lambda$.
 This is equivalent to the existence of a \emph{flex}
 for some realization~$\rho$ inducing~$\lambda$, namely, a continuous function
 \[
  f \colon [0,1) \longrightarrow \bigl(\R^3\bigr)^{V_{\Bipyramid}}
 \]
 such that
 \begin{itemize}
  \item $f(0) = \rho$;
  \item $f(t)$ is a realization of~$\Bipyramid$
  inducing the same edge lengths as $\rho$ for any $t \in [0,1)$;
  \item $f(t_1)$ and $f(t_2)$ are not congruent for any $t_1 \neq t_2$.
 \end{itemize}
\end{definition}

From \Cref{definition:flexible} we see that it can be useful to consider the set of realizations of a bipyramid
inducing the same edge lengths, up to the action of the group of direct isometries of~$\R^3$.
We propose to use a notion of \emph{configuration space} derived from the classical construction of \emph{Cayley-Menger varieties}, similarly as Borcea did in \cite{Borcea2002}.

We introduce a variety that keeps track of the distances between the vertices of~$\Bipyramid$ and (in addition to the classical Cayley-Menger variety) also of the oriented volume of the tetrahedra
determined by the vertices $(i, i+1, \Top, \Bottom)$.
Recall that the \emph{oriented volume} of the tetrahedron in $\R^3$
defined by the points $(a_1, a_2, a_3, a_4)$ is
\[
 \frac{1}{6} \det( a_1 - a_4 \quad a_2 - a_4 \quad a_3 - a_4 ) \,.
\]

\begin{definition}
 \label{definition:oriented_cayley_menger}
 Let $\Eq$ be the set of equatorial edges of~$\Bipyramid$,
 where we choose the following orientation
 \[
  \bigl\{ (1, 2), \dotsc, (n-1, n), (n, 1) \bigr\} \,.
 \]
 Consider the map:
 \[
  (\R^3)^{V_{\Bipyramid}} \rightarrow \R^{\binom{V_{\Bipyramid}}{2}} \times \R^{\Eq}
 \]
 sending a realization $\rho$ of~$\Bipyramid$ to the pair
 \[
   \Bigl( \left(d_{vw}\right)_{v, w \in V_{\Bipyramid}}, (s_e)_{e \in \Eq} \Bigr)\,,
 \]
 where $d_{vw} = \left\| \rho(v) - \rho(w) \right\|^2$ and
 for $e = (i, i+1)$, the number $s_e$ is the oriented volume
 of the tetrahedron $(i, i+1, \Top, \Bottom)$.
 Here, whenever $i = n$, we consider $i+1$ to be $1$.
 For convenience, we may write $s_{i,i+1}$ instead of $s_{(i,i+1)}$.
 We define the \emph{oriented Cayley-Menger variety} $\OCM$ to be the Zariski closure of the image of this map.
 By construction, the function field $\R(\OCM)$ is generated by the classes of the functions $\left(d_{vw}\right)_{v, w \in V_{\Bipyramid}}, (s_e)_{e \in \Eq}$.

 Given edge lengths~$\lambda$ for~$\Bipyramid$, we define the \emph{configuration space}
 \[
  K_{\lambda} := \OCM \cap \bigl\{ d_{vw} = \lambda^2_{\{v,w\}} \text{ for all } \{v,w\} \in E_{\Bipyramid} \bigr \} \,.
 \]
\end{definition}

We introduce some quantities that play a crucial role in \Cref{galois}.

\begin{definition}
\label{definition:volumes}
 For a realization $\rho$ of $\Bipyramid$, we denote by $\varsigma_{ij}$ the oriented volume of the realization of a tetrahedron $(i, j, \Top, \Bottom)$.
 In particular, we have $\varsigma_{i, i+1} = s_{i,i+1}$ and $\varsigma_{i,j} = - \varsigma_{j,i}$. \Cref{equations} reports equations relating the quantities $\varsigma_{i,j}$ and the distances~$d_{uv}$.
\end{definition}

\begin{remark}
  Any configuration space of~$\Bipyramid$ is a curve since, if it was at least two-dimensional,
 then we could fix the distance between~$\Top$ and~$\Bottom$ and still have a one-dimensional motion,
 but this is impossible because the structure we obtain is composed of tetrahedra with fixed edge lengths.
\end{remark}
\begin{definition}
\label{definition:motion}
 A \emph{motion} of an $n$-gonal bipyramid is any one-dimensional component
 of the configuration space $K_{\lambda}$ of~$\Bipyramid$ with fixed edge lengths~$\lambda$
 whose general point is given by the distances induced by a realization of~$\Bipyramid$ inducing~$\lambda$.
 We call a motion \emph{non-degenerate} if the vertices of each $3$-cycle in~$\Bipyramid$ are non-collinear
 and all non-edges change their lengths along the motion.
\end{definition}

Notice that it is not possible that the dihedral angle at an edge of the equator stays constant along any motion
(equivalently, $s_{i,i+1}$ is constant)
since this would fix the distance between~$\Top$ and~$\Bottom$, which implies all dihedral angles being constant.
Instead, if the dihedral angle at an edge of the form $\{ \Top, i \}$ or $\{ \Bottom, i \}$,
where $i$ is a vertex of the equator, stays constant along a motion,
then the distance between the vertices $i-1$ and $i+1$ (mod $n$) is constant and
so we fall back to the case of a flexible $\mathbb{B}_{n-1}$ by omitting the vertex~$i$ and adding the edge $\{ i-1, i+1 \}$.
This is why we ask this not to happen for non-degenerate motions.

For any non-degenerate motion~$\motion$ and almost tetrahedron~$\Tetrahedron{i}{i+1}$,
there is a regular algebraic map
\[
 \pi_{i,i+1} \colon \motion \longrightarrow \motion_{i,i+1}\,,
\]
where the curve $\motion_{i,i+1}$ is a subset of
the configuration curves $\Tetrahedron{i}{i+1}$,
namely the general points in $\motion_{i,i+1}$ are the restrictions to~$\Tetrahedron{i}{i+1}$
of distances determined by realizations of~$\Bipyramid$.
The fact that~$\motion_{i, i+1}$ is a curve (and not a point)
is a consequence of the non-degeneracy assumption on~$\motion$.

Notice that, in order to freeze the action of the group of direct isometries, one could simply pin down the realization of a triangle in~$\Bipyramid$:
in this way, any two different realizations with the same pinned triangle would automatically be not congruent under any direct isometry.
However, this approach would lead to some technical subtleties to be handled when, for example,
we have to consider at the same time the configuration spaces of two almost tetrahedra that do not share a triangle.
Because of this, we prefer to use the approach exposed above,
which allows us to easily restrict from the configuration space of the whole~$\Bipyramid$
to the one of an almost tetrahedron.

\subsection{Flexible octahedra and the Nelson construction}
\label{basics:octahedra}

It is well-known from the works of Bricard (see \cite{Bricard1897}),
that there are three kinds of motions for octahedra:
\begin{itemize}
 \item Type I (line-symmetric): here, in each realization, there are three pairs of vertices that are symmetric with respect to a line;
 \item Type II (plane-symmetric): here, in each realization, the set of vertices is symmetric with respect to a plane; moreover, two vertices lie on the symmetry plane;
 \item Type III: here, the edge lengths in each of the three induced $4$-cycles satisfy a linear relation
 	and there are two flat poses.
\end{itemize}

We recall a construction idea from Nelson \cite{Nelson2010},
that allows us to form a flexible pentagonal bipyramid starting from two flexible octahedra.

Let $\Oct_\kasperl$ and $\Oct_\pezi$ be two octahedra, i.e., $4$-gonal bipyramids.
Define $V_{\Oct_\kasperl} := \{1,2,3,0,\Top,\Bottom\}$ and $V_{\Oct_\pezi} := \{3,4,5,0,\Top,\Bottom\}$ so that
the non-edges in~$\Oct_\kasperl$ are $\{1,3\}$, $\{2,0\}$, $\{\Top,\Bottom\}$ and
the non-edges in~$\Oct_\pezi$ are $\{3,5\}$, $\{4,0\}$, $\{\Top,\Bottom\}$.
A pair~$(\rho_\kasperl,\rho_\pezi)$ of realizations of the two octahedra is called a \emph{fitting pair}
if the following hold:
\begin{itemize}
 \item $\rho_\kasperl(i)=\rho_\pezi(i)$ for all $i \in \{0,3,T,B\}$;
 \item the three points $\rho_\kasperl(0),\rho_\kasperl(1),\rho_\pezi(5)$ are collinear.
\end{itemize}
A fitting pair induces a realization $\rho_\Penta$ of the pentagonal bipyramid with vertices
$1$, $2$, $3$, $4$, $5$, $\Top$, $\Bottom$:
for each vertex~$u$, we define~$\rho_\Penta(u)$ to be
$\rho_\kasperl(u)$ or $\rho_\pezi(u)$,
depending which of these two values is defined.
If both values are defined, then they coincide by the definition of fitting pair,
so $\rho_\Penta$ is well-defined.
Note that the distance between~$\rho_\Penta(1)$
and~$\rho_\Penta(5)$ is determined by the induced edge lengths of the octahedra,
up to only two possible values, namely $\lambda^{\kasperl}_{\{1,0\}} + \lambda^{\pezi}_{\{5,0\}}$
and~$|\lambda^{\kasperl}_{\{1,0\}} - \lambda^{\pezi}_{\{5,0\}}|$.

Assume now that a fitting pair $(\rho_\kasperl,\rho_\pezi)$ admits a \emph{flex}, namely a pair $(f_\kasperl, f_\pezi)$ of flexes
\[
 f_\kasperl \colon [0,1) \longrightarrow \bigl(\R^3\bigr)^{V_{\Oct_\kasperl}}
 \quad
 \text{and}
 \quad
 f_\pezi \colon [0,1) \longrightarrow \bigl(\R^3\bigr)^{V_{\Oct_\pezi}}
\]
such that $\bigl( f_\kasperl(t), f_\pezi(t) \bigr)$ is a fitting pair for any $t \in [0,1)$.

When a fitting pair admits a flex,
by simply forgetting vertex~$0$ we obtain a flex of a pentagonal bipyramid.
This is what we call a \emph{Nelson construction} of a flexible pentagonal bipyramid,
see also \Cref{fig:nelson}.

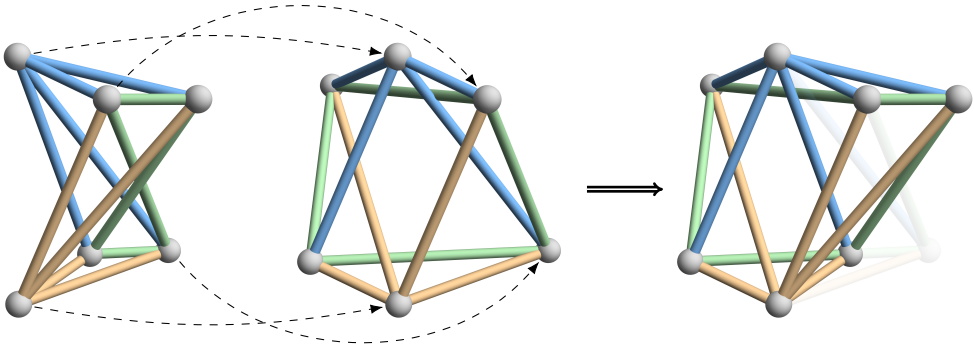
\begin{figure}[ht]
 \begin{tikzpicture}[inode/.style={circle,minimum size=0.4cm},glue/.style={dashed,-{Latex[]}}]
  \node (o1) at (0,0) {\includegraphics[height=4cm]{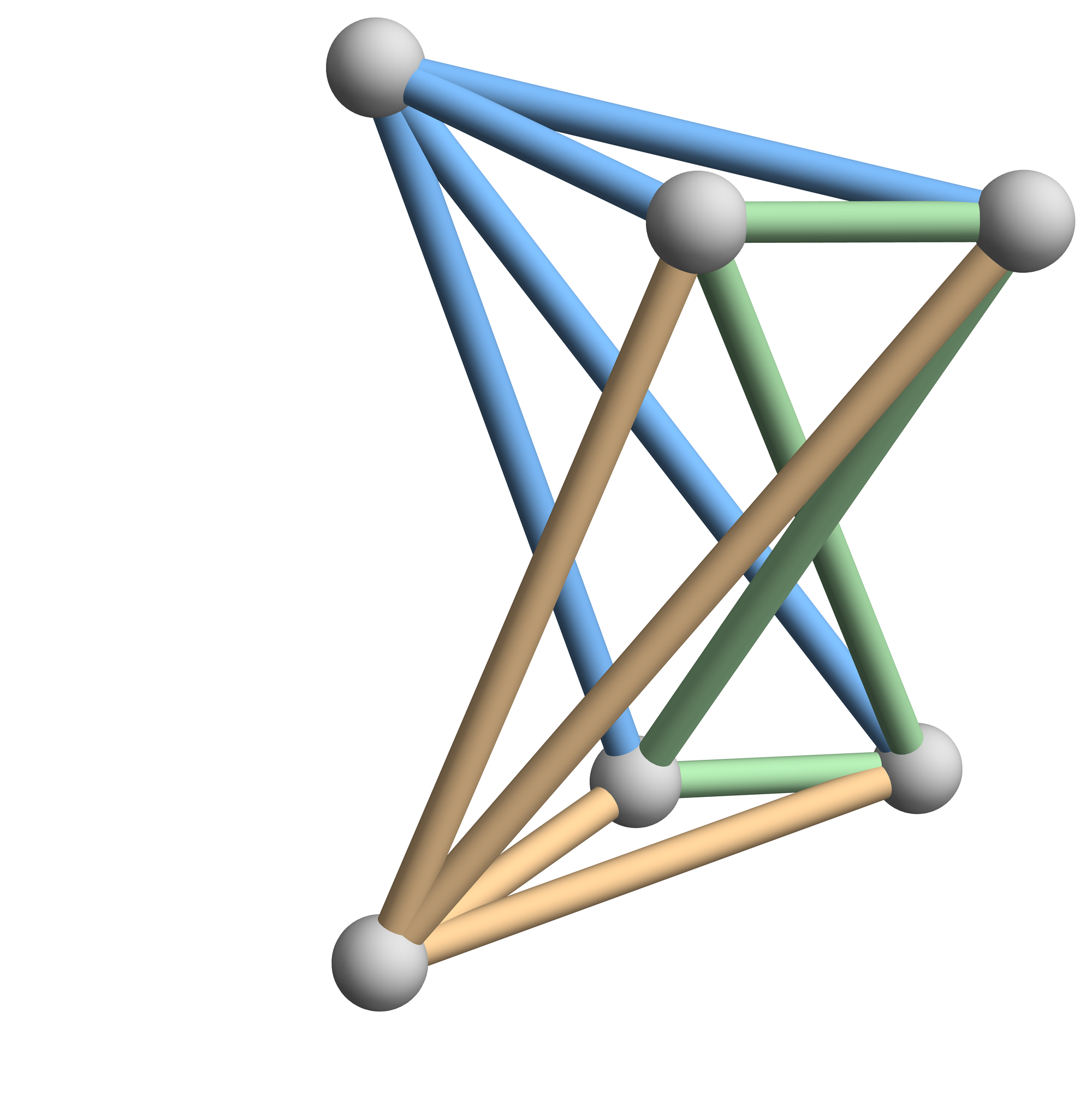}};
  \node (o2) at (5,0) {\includegraphics[height=4cm]{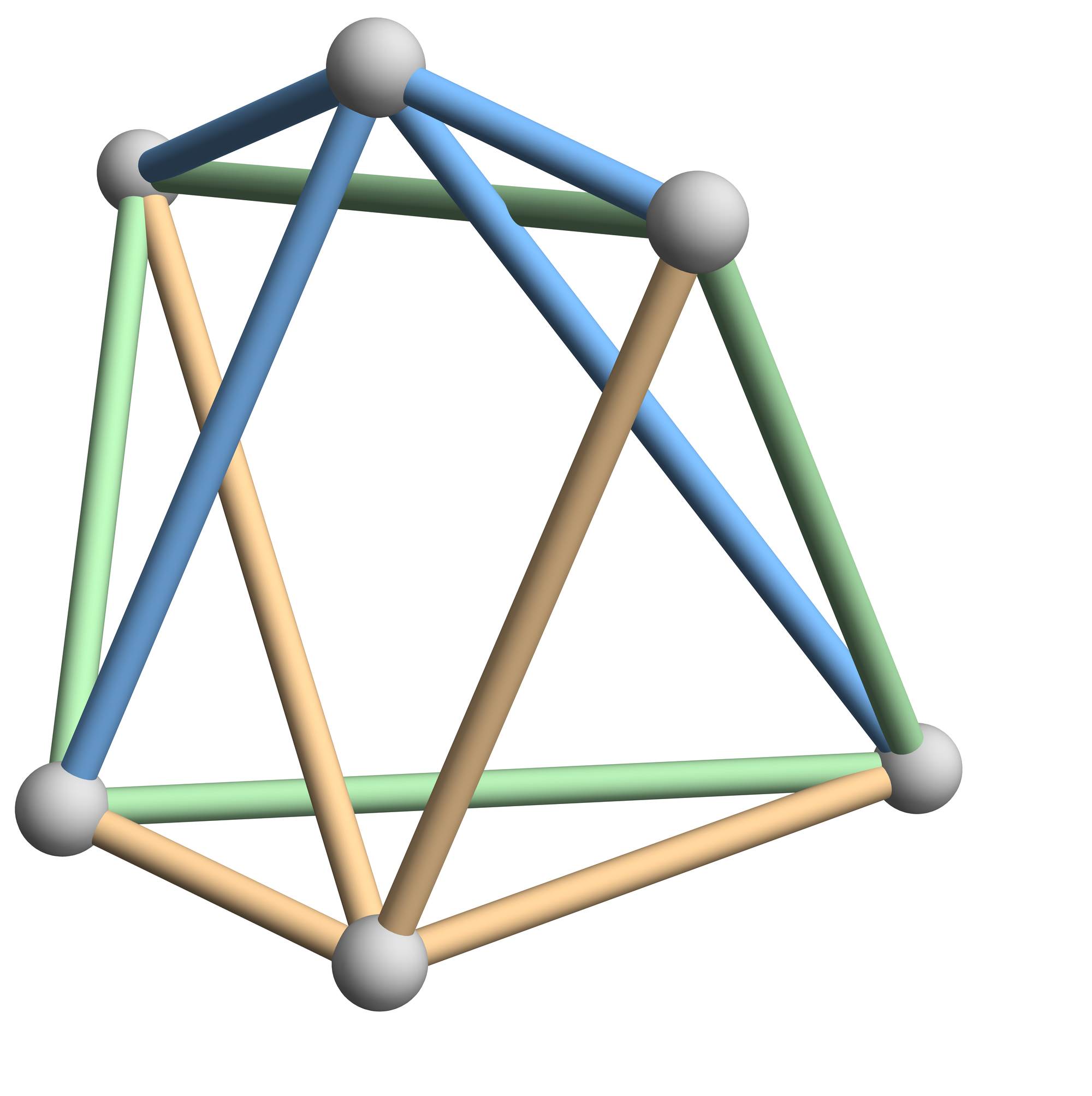}};
  \node[] (a) at (10,0) {\includegraphics[height=4cm]{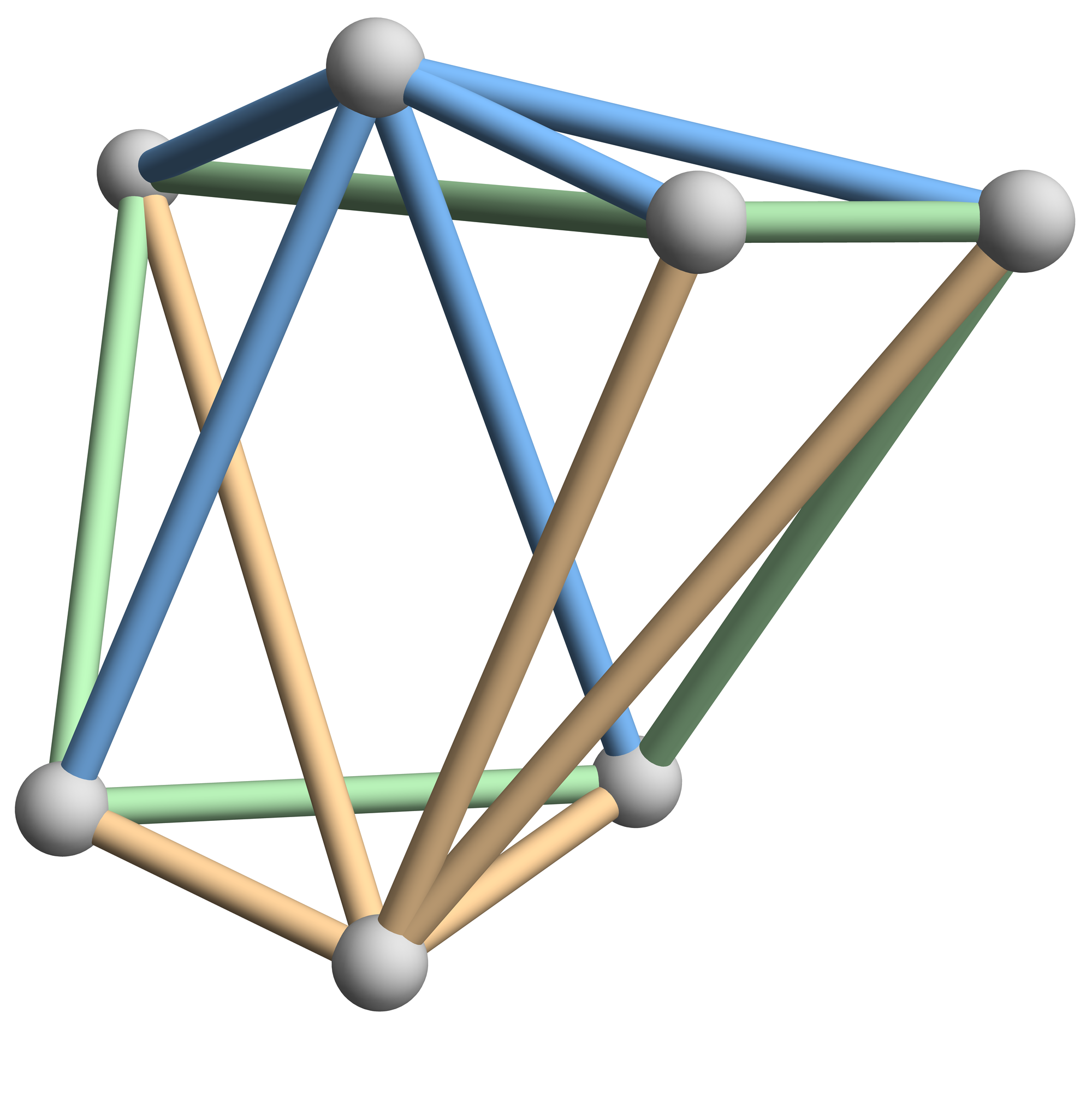}};
  \node[opacity=0.1] (a) at (10,0) {\includegraphics[height=4cm]{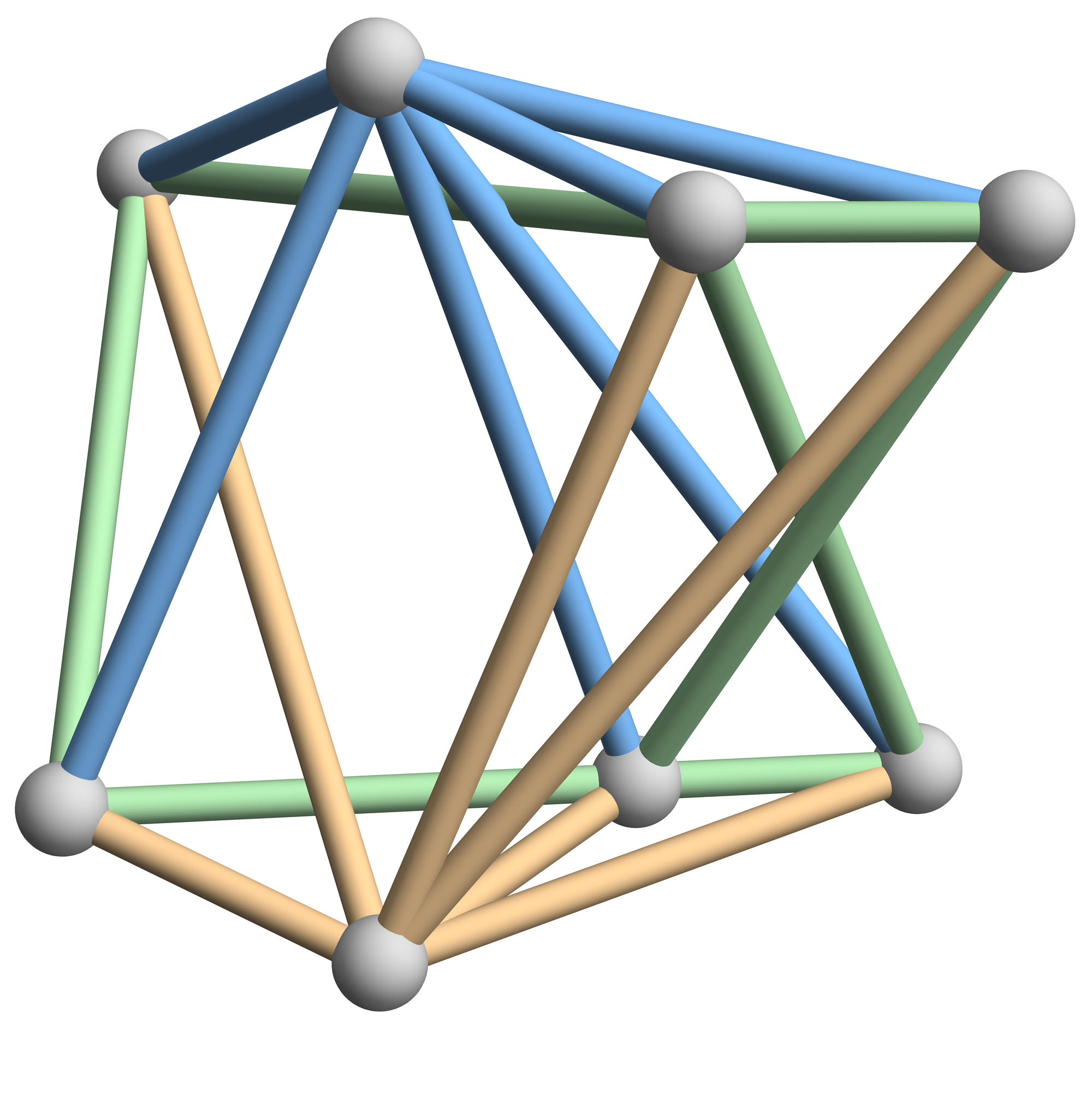}};
  \node[inode] (na1) at (-0.61,-1.525) {};
  \node[inode] (na2) at ($(na1)+(o2)$) {};
  \node[inode] (nb1) at (-0.62,1.765) {};
  \node[inode] (nb2) at ($(nb1)+(o2)$) {};
  \node[inode] (nc1) at (0.55,1.2) {};
  \node[inode] (nc2) at ($(nc1)+(o2)$) {};
  \node[inode] (nd1) at (1.37,-0.8) {};
  \node[inode] (nd2) at ($(nd1)+(o2)$) {};
  \draw[glue] (na1) to[bend right=10] (na2);
  \draw[glue] (nb1) to[bend left=10] (nb2);
  \draw[glue] (nc1) to[bend left=50] (nc2);
  \draw[glue] (nd1) to[bend right=50] (nd2);
  \draw[line width=1pt,double,-{Implies[]}] ($(a.west)-(1,0)$)--(a.west);
 \end{tikzpicture}
 \caption{The idea of the Nelson construction.}
 \label{fig:nelson}
\end{figure}

\section{Galois theory for motions}
\label{galois}

In this section, we define the notion of \emph{Galois group} of a motion of a pentagonal bipyramid and we describe which groups arise as such.
This notion, however, can be defined for arbitrary $n$-gonal bipyramids.
Since the first results that we obtain are true for arbitrary bipyramids, we start with the more general setting and then we specialize to the case of pentagonal bipyramids.
We point out that some of the results in this section (\Cref{lemma:oriented_volume_zero_sum} and \Cref{corollary:flat_poses_two_ones}) resamble the ones by Connelly in \cite[Lemma~3 and~4]{Connelly1974} (see also \cite[Section~2]{Alexandrov2011}): however, the two techniques are different and a translation between them is not apparent, although we believe it would be worth of investigation.

We fix a non-degenerate motion~$\motion$ of an $n$-gonal bipyramid and we consider the rational map
\[
 \distance \colon \motion \dashrightarrow \R
\]
associating to a point of~$\motion$ its $\Top\Bottom$-coordinate, i.e., the squared distance between the realizations of~$\Top$ and~$\Bottom$.
This rational map determines a field extension $\R(d_{\Top\Bottom}) \subset \R(\motion)$,
where $\R(d_{\Top\Bottom})$ and $\R(\motion)$ are the function fields of the affine line and~$\motion$, respectively.

\begin{definition}
 We define $\Gal$ to be the Galois group of the extension $\R(d_{\Top\Bottom}) \subset \R(\motion)$.
\end{definition}

We first show that $\R(d_{\Top\Bottom}) \subset \R(\motion)$ is Galois (\Cref{lemma:galois}) and
that $\Gal$ can be considered as a subgroup of~$\Z_2^{\Eq}$ (\Cref{lemma:embedding}).
The latter embedding uses in a key way the restriction of the motion~$\motion$
to each of the $n$ almost tetrahedra constituting the $n$-gonal bipyramid.
We then establish several relations between the existence of certain elements in the Galois group
and the oriented volumes of tetrahedra with vertices in~$\Bipyramid$,
together with their consequences on flexible bipyramids.
After that, we specialize to pentagonal bipyramids and we explore how kinematics and geometry constrain the possible elements in the Galois group,
until we prove that there are only two possibilities (\Cref{proposition:galois_groups}).
We conclude by showing that each of these groups can be realized by a Nelson construction (\Cref{theorem:nelson_Z2}, \Cref{theorem:nelson_opposite}).

Since $\motion$ is a subvariety of $\OCM$, all the coordinate functions $d_{vw}$ and $s_{i,i+1}$ are elements of $\R(\motion)$.
Keep in mind that, for every edge $e = \{v, w\} \in E_{\Bipyramid}$, we have the equation $d_e = \lambda_e^2$ in $\motion$, hence the coordinate function $d_e$ is invariant under any element of~$\Gal$.
We start by clarifying how the Galois group $\Gal$ acts on the $s$-variables.

\begin{proposition}
\label{proposition:galois_on_volume}
 Let $\varphi \in \Gal$. Then for every $i \in \{1, \dotsc, n\}$, we have
 \[
  \varphi(s_{i,i+1}) = \pm s_{i,i+1} \,.
 \]
\end{proposition}
\begin{proof}
 By \Cref{lemma:equation_oriented_volume}, $s_{i,i+1}$ satisfies an equation of the form $s_{i,i+1}^2 - F = 0$ where $F$ is a polynomial in the variables $(d_{vw})_{\{v,w\} \in E_{\Bipyramid}}$ and $d_{\Top\Bottom}$.
 Therefore, the element $F$ is left invariant by~$\varphi$, hence
 \[
  \varphi(s_{i,i+1})^2 = F \,,
 \]
 which gives the statement.
\end{proof}

As we mentioned, the study of the restriction of the motion~$\motion$ to the various almost tetrahedra
is crucial in our understanding of the Galois group.
Recall that the curve~$\motion_{i,i+1}$ is the image of~$\motion$ under the restriction map~$\pi_{i,i+1}$
between the configuration space of $\Bipyramid$ and the configuration space of the almost tetrahedron~$\Tetrahedron{i}{i+1}$.
Then, the map~$\distance$ fits into a commutative diagram
\begin{center}
 \begin{tikzpicture}
  \node[] (m) at (-1.6,0) {$\motion$};
  \node[] (mp) at (0,1.25) {$\motion_{i,i+1}$};
  \node[] (c) at (1.6,0) {$\R$};
  \draw[cdiagd] (m)to node[below,labelsty] {$\distance$} (c);
  \draw[cdiagd] (mp)to node[anchor=south west,labelsty] {$(\distance)_{i,i+1}$} (c);
  \draw[cdiag] (m)to node[anchor=south east,labelsty] {$\pi_{i,i+1}$} (mp);
 \end{tikzpicture}
\end{center}
where $(\distance)_{i,i+1}$ is the distance (from~$\Top$ to~$\Bottom$) map applied to realizations in~$\motion_{i,i+1}$.
Notice that $\motion_{i,i+1}$ is irreducible: its elements are configurations of the almost tetrahedron~$\Tetrahedron{i}{i+1}$,
which are parametrized by the angle between the triangles $\{i, i+1, \Top\}$ and $\{i, i+1, \Bottom\}$.
The maps~$\pi_{i,i+1}$ are dominant: in fact, $\pi_{i,i+1}$ is non-constant since if the distance between~$\Top$ and~$\Bottom$
changes, then the angle in the almost tetrahedron changes as well;
in the same way, we show that the maps $(\distance)_{i,i+1}$ are dominant as well.
Hence, we can interpret the function fields $\R(\motion_{i,i+1})$
as intermediate field extensions between~$\R(d_{\Top\Bottom})$ and~$\R(\motion)$.
Moreover, we call the Galois group of $\R(d_{\Top\Bottom}) \subseteq \R(\motion_{i,i+1})$ the Galois group of the almost tetrahedron~$\Tetrahedron{i}{i+1}$.

\begin{lemma}
\label{lemma:quadratic_extension}
 Each extension $\R(d_{\Top\Bottom}) \subseteq \R(\motion_{i,i+1})$ is of degree~$2$.
 The Galois group of~$\Tetrahedron{i}{i+1}$ is~$\Z_2$.
\end{lemma}
\begin{proof}
 All the $d$-coordinate functions in $\R(\motion_{i, i+1})$ are left invariant by the Galois group.
 Due to \Cref{proposition:galois_on_volume}, the only $s$-coordinate function, namely $s_{i,i+1}$, satisfies a degree~$2$ polynomial equation over $\R(d_{\Top\Bottom})$.
 If $s_{i,i+1}$ satisfied a degree~$1$ equation over $\R(d_{\Top\Bottom})$, then the complexification of the map $d_{\Top\Bottom} \colon \motion_{i,i+1} \rightarrow \R$ would be of degree~$1$, i.e., a general element of the codomain would have a preimage of cardinality~$1$.
 This, however, would contradict the fact that, $\Tetrahedron{i}{i+1}$ being flexible,
 there is a whole interval of values for the distance between~$\Top$ and~$\Bottom$
 such that there exist two realizations with the same edge lengths and the given distance $d_{\Top\Bottom}$ that are not related by a direct isometry.
 Hence, $\R(d_{\Top\Bottom}) \subseteq \R(\motion_{i,i+1})$ has degree~$2$, so the Galois group of~$\Tetrahedron{i}{i+1}$ is~$\Z_2$.
\end{proof}

\begin{lemma}
\label{lemma:generation_subfields}
 The field $\R(\motion)$ is generated by the subfields $\{ \R(\motion_{i,i+1}) \}_{i=1}^{n}$.
\end{lemma}
\begin{proof}
 To show that $\R(\motion)$ is generated by $\{ \R(\motion_{i,i+1}) \}_{i=1}^{n}$,
 it is enough to prove that all the coordinate functions of~$\motion$ can be written as rational functions
 in the coordinate functions of each~$\motion_{i,i+1}$.
 Recalling \Cref{definition:oriented_cayley_menger},
 we then have to express all distances~$d_{vw}$, where $v,w \in V_{\Bipyramid}$,
 as rational functions in the distances~$d_{v'w'}$,
 where $v',w' \in V_{\Tetrahedron{i}{i+1}}$, and the coordinate functions $s_{i,i+1}$
 for some $i \in \{1, \dotsc, n\}$ and the distance~$\distance$. This is a special case of \Cref{corollary:rational_functions}.
\end{proof}

\begin{lemma}
\label{lemma:galois}
 The extension $\R(d_{\Top\Bottom}) \subset \R(\motion)$ is Galois.
\end{lemma}
\begin{proof}
 Each extension $\R(d_{\Top\Bottom}) \subset \R(\motion_{i,i+1})$ is quadratic by \Cref{lemma:quadratic_extension},
 so $\R(\motion_{i,i+1})$ is a splitting field of a polynomial $x^2 - \alpha_{i, i+1}$ with $\alpha_i \in \R(d_{\Top\Bottom})$: its roots are the elements $s_{i,i+1}$ and $-s_{i,i+1}$.
 Hence, by \Cref{lemma:generation_subfields} the field~$\R(\motion)$ is a splitting field of the polynomial $\prod_{i=1}^{n} (x^2 - \alpha_i)$.
 Therefore, the extension $\R(d_{\Top\Bottom}) \subset \R(\motion)$ is Galois.
\end{proof}

\begin{lemma}
\label{lemma:embedding}
 The Galois group $\Gal$ is a subgroup of~$\Z_2^{\Eq}$,
 where $\Eq$ is the set of equatorial edges of $\Bipyramid$.
\end{lemma}
\begin{proof}
 Let $H_{i,i+1}$ be the subgroup of~$\Gal$ determined by~$\R(\motion_{i,i+1})$ via the Galois correspondence.
 Since $\R(d_{\Top\Bottom}) \subset \R(\motion_{i,i+1})$ is quadratic,
 the subgroup $H_{i,i+1}$ is normal in~$\Gal$ and $\Gal/H_{i,i+1} \cong \Z_2$.
 Since $\R(\motion)$ is generated by $\{ \R(\motion_{i,i+1}) \}_{i=1}^{n}$,
 the intersection of all $\{ H_{i,i+1} \}_{i=1}^{n}$ is the trivial subgroup.
 This implies that the map
 \[
  \begin{array}{rcl}
   \Gal & \longrightarrow & \Gal/H_{1,2} \times \dotsb \times \Gal/H_{n,1}\,, \\
   \varphi & \longmapsto & ([\varphi]_{1,2} , \dotsc, [\varphi]_{n,1})
  \end{array}
 \]
 is an injective homomorphism.
\end{proof}

\begin{definition}
\label{definition:embedding}
 From now on, we identify elements of~$\Gal$ with their images in $\Z_2^{\Eq}$.
 As clarified by the proof of \Cref{lemma:embedding},
 each factor of $\Z_2^{\Eq}$ corresponds to an almost tetrahedron.
 Hence, we represent an element~$\varphi$ in $\Gal$ by a binary tuple
 such that the entry $(i,i+1)$ is $1$
 if and only if $\varphi(s_{i,i+1}) = -s_{i,i+1}$.
 We denote by $\ones{\varphi}$ the set of equatorial edges $e$
 such that the entry $e$ in the tuple of $\varphi$ is 1.
\end{definition}

\begin{lemma}
\label{lemma:element11}
 For every two equatorial edges $e_1, e_2 \in \Eq$, there exists $\varphi \in \Gal$ such that $e_1, e_2\in \ones{\varphi}$.
\end{lemma}
\begin{proof}
 For each equatorial edge~$e$, the projection $\Gal \longrightarrow \Z_2$ onto the factor corresponding to~$e$ is surjective.
 Indeed, by the fundamental theorem of Galois theory, the image of that projection is the Galois group of $\mathbb{T}_{e}$,
 which is $\Z_2$ by \Cref{lemma:quadratic_extension}.
 Let $\varphi, \psi \in \Gal$ be such that $e_1,e_2\in\ones{\varphi}$. Then, at least one of $\varphi$, $\psi$, or $\varphi + \psi$ is a witness for the statement.
\end{proof}

The next series of results focuses on the relation between
the Galois group and the oriented volumes of tetrahedra in~$\Bipyramid$.

\begin{lemma}
\label{lemma:consecutive_ones}
 If $\varphi \in \Gal$ is such that $\ones{\varphi}$ are consecutive equatorial edges,
 namely, $\varphi$ is of the form
 \[
  (0, \dotsc, 0, \underbrace{\overset{(i, i+1)}{1}, \dotsc, \overset{(j-1, j)}{1}}_{(j-i) \text{ ones}}, 0, \dotsc, 0) \,,
 \]
 then $\varsigma_{i,j} = 0$. Recall from \Cref{definition:volumes} and \Cref{corollary:rational_functions} that $\varsigma_{i,j}$ is the rational function in $\R(\motion)$ representing the volume of the tetrahedron $(i,j,\Top, \Bottom)$. In particular, the vertices $i, j, \Top, \Bottom$ are coplanar during the motion.
\end{lemma}
\begin{proof}
  Let $\psi \in \Gal$.
  We first prove that if $(i, i+1)$ and $(i+1, i+2)$ are in $\ones{\psi}$, then $\psi(\varsigma_{i, i+2}) = -\varsigma_{i, i+2}$.
  In fact, from \Cref{lemma:equation_distance_volume} it follows that $\psi(d_{13}) = d_{13}$.
  Now, using \Cref{lemma:oriented_volume_diagonal} with $k = i+1$ and $j=i+2$ the statement follows,
  since $\psi(a_{ijk}) = -a_{ijk}$ while $\psi(b_{ijk}) = b_{ijk}$,
  because $\psi$ leaves the squared volumes $W$ invariant thanks to the fact that $\psi(d_{13}) = d_{13}$.
  By the very same argument (and using induction as in \Cref{corollary:rational_functions}), when all the entries from $(i, i+1)$ until $(j-1, j)$ are ones, we have $\psi(\varsigma_{i, j}) = -\varsigma_{i, j}$.
  Similarly, if all entries from $(i, i+1)$ until $(j-1, j)$ of $\psi$ are zero,
  one has $\psi(\varsigma_{i, j}) = \varsigma_{i, j}$,
  since the only difference from the previous case is that $\psi(a_{ijk}) = a_{ijk}$.
  Hence, for the element $\varphi$ from the statement, we have $\varphi(\varsigma_{i,j}) = -\varsigma_{i,j}$ considering the ones from entry $(i,i+1)$ to entry $(j-1,j)$, and at the same time we have $\varphi(\varsigma_{i,j}) = \varsigma_{i,j}$ considering the zeros from entry $(j,j+1)$ to entry $(i-1,i)$.
  Altogether, we obtain $\varsigma_{i,j} = 0$.
  Since $\varsigma_{i,j}$ is the a volume of the tetrahedron $(i, j, \Top, \Bottom)$, this means that such volume is zero for a general element of the motion, hence the four vertices $i, j, \Top, \Bottom$ stay coplanar during the motion.
\end{proof}

Since we suppose that the motion~$\motion$ is non-degenerate, each tetrahedron $(i, i+1, \Top, \Bottom)$ changes its volume during the motion. In view of \Cref{lemma:consecutive_ones}, we obtain:

\begin{corollary}
\label{corollary:no_one}
 No element in~$\Gal$ has a single one or a single zero.
\end{corollary}

\begin{lemma}
\label{lemma:oriented_volume_zero_sum}
	If $\varphi\in\Gal$, then
	\[
		\sum_{(i,i+1)\in\ones{\varphi}} s_{i,i+1} = 0\,.
	\]
\end{lemma}
\begin{proof}
	We use the fact that the generalized volume of a flexible $n$-gonal bipyramid
	equals zero~\cite{Connelly1978}, but the Bellows Theorem~\cite{bellows}
	stating that it is a constant would suffice as well.

	For a fixed orientation of the faces of the polyhedron
	(i.e., a cyclic ordering of the vertices of each face
	such that each edge is ordered in the opposite direction in the two faces it belongs to),
	the generalized volume of the polyhedron with realization $\rho$ is defined by
	\[
		\frac{1}{6} \sum \det \bigl( \rho(i), \rho(j), \rho(k) \bigr)\,,
	\]
	where the sum is over all oriented faces $(i,j,k)$.

	We fix an orientation of the faces of~$\Bipyramid$ as follows:
	$(i, i+1, \Top)$ and $(i+1, i, \Bottom)$ for $i \in \{1, \dotsc, n\}$.
	Since the generalized volume is invariant under direct isometries,
	we can suppose that $\Bottom$ is at the origin.
	Hence, the faces containing $\Bottom$ do not contribute to the generalized volume
	and the faces containing~$\Top$ supply the oriented volumes of the tetrahedra $(i, i+1, \Top, \Bottom)$.
	So the generalized volume of an $n$-gonal bipyramid is
	\[
		W = \sum_{i\in\{1,\ldots,n\}} s_{i,i+1} = 0\,.
	\]
	Taking the difference $W - \varphi(W)$ yields the statement.
\end{proof}

\begin{corollary}
\label{corollary:flat_poses_two_ones}
	If $\varphi\in\Gal$ has exactly two ones, then the two almost tetrahedra corresponding to edges in $\ones{\varphi}$
	flatten simultaneously at two poses.
\end{corollary}

\begin{corollary}
\label{corollary:edge_lengths_two_ones}
	If $\varphi\in\Gal$ has exactly two ones,
	then the two edges in $\ones{\varphi}$ have the same length.
\end{corollary}
\begin{proof}
	Let $\ones{\varphi} = \{(i,i+1), (j,j+1)\}$.
	By \Cref{lemma:oriented_volume_zero_sum}, we know that $s_{i,i+1} = -s_{j,j+1}$, hence $s_{i,i+1}^2 = s_{j,j+1}^2$.
	By \Cref{lemma:equation_oriented_volume}, the determinants of the matrices $M_{i,i+1,\Top,\Bottom}$ and $M_{j,j+1,\Top,\Bottom}$ (introduced in the lemma) are equal. These determinants are polynomials in $d_{\Top\Bottom}$, and comparing their leading coefficients yields the result.
\end{proof}

\begin{corollary}
	\label{corollary:no_single_one_simultaneously}
	Let $\varphi,\psi\in\Gal$.
	Then
	\[
		\sum_{(i,i+1)\in \ones{\varphi} \cap \ones{\psi}} s_{i,i+1} = 0\,.
	\]
	In particular, $|\ones{\varphi}\cap \ones{\psi}|\neq 1$.
\end{corollary}
\begin{proof}
	\Cref{lemma:oriented_volume_zero_sum}
	applied to $\varphi$, $\psi$, and $\varphi + \psi$,
	yields
	\[
		\sum_{(i,i+1)\in \ones{\varphi}} s_{i,i+1} = 0\,, \qquad
		\sum_{(i,i+1)\in \ones{\psi}} s_{i,i+1} = 0\,, \qquad
		\sum_{(i,i+1)\in S} s_{i,i+1} = 0\,,
	\]
	where $S = (\ones{\varphi}\cup \ones{\psi}) \setminus (\ones{\varphi}\cap \ones{\psi})$.
	Subtracting the third equation from the sum of the first two gives the first statement.
	The second one follows since no $s_{i,i+1}$ can be constantly zero.
\end{proof}

The simplest possible Galois group arises only for special motions.

\begin{proposition}
\label{proposition:Z2_implies_flat poses}
	If $\Gal = \langle (1,1,\ldots,1) \rangle$,
	then there are two flat poses of the whole bipyramid.
\end{proposition}
\begin{proof}
  The proof is in the same spirit as \Cref{lemma:quadratic_extension}:
  we show that if there is at most one flat pose,
  then the Galois group $\Gal$ has more than $2$ elements,
  by inferring that the complexification of the map $\distance \colon \motion \dashrightarrow \R$ has degree greater than $2$.

  For each equatorial edge $e \in \Eq$,
  let $m_e$ be the minimal distance between $\Top$ and $\Bottom$ obtainable in a realization of $\mathbb{T}_e$,
  that induces the edge lengths of the motion under consideration;
  likewise, let $M_e$ be the maximal such distance.
  These two distances correspond to the two flat poses of~$\mathbb{T}_e$.
  Notice that it is possible that neither $m_e$ nor $M_e$ is in the image of~$\distance$.
  Define now $m = \max\{m_e\}_{e \in \Eq}$ and $M = \min\{M_e\}_{e \in \Eq}$.
  Notice that the interval $(m, M)$ is not empty since $\motion$ is a motion according to \Cref{definition:motion}.

  If $m_e = m$ for all $e \in \Eq$,
  then the whole bipyramid has a flat pose with $\distance = m^2$.
  Hence, if the whole bipyramid does not have a flat pose with $\distance = m^2$,
  then some $m_{\bar{e}}$ is smaller than $m$.
  This implies that $\mathbb{T}_{\bar{e}}$ is not flat when $\distance = m^2$, since $m_{\bar{e}} < m < M \leq M_{\bar{e}}$.
  Let $e' \in \Eq$ be such that $m_{e'} = m$.

  We claim that $m^2 \in \distance(\motion_{e'})$.
  To show this, we exhibit a configuration $x \in \motion$ such that $\distance(x) = m^2$.
  Consider the complexification~$(\distance)_{\mathbb{C}}$ of the map $\distance \colon \motion \dashrightarrow \R$.
  The image~$Z$ of the latter map equals~$\mathbb{C}$ with at most finitely many points removed.
  Hence, there exists $\varepsilon > 0$ such that $(m^2, m^2 + \varepsilon) \subset Z$.
  For each $t \in (m^2, m^2 + \varepsilon)$, let $x_t \in (\distance)_{\mathbb{C}}^{-1}(t)$.
  We argue that actually $x_t$ is a real point.
  Indeed, all the values of $\{d_{e}\}_{e \in E_{\Bipyramid}}$ in~$x_t$ are real by construction.
  Moreover, since $\distance(x_t) = t$ and $m^2 \leq t \leq M^2$,
  \Cref{lemma:equation_oriented_volume} shows that all the values of $\{s_e\}_{e \in \Eq}$ in~$x_t$ are real.
  Then \Cref{corollary:rational_functions} implies that all the values $\{d_{uv}\}_{u, v \in V_{\Bipyramid}}$ in~$x_t$ are real, making $x_t$ a real point.
  Since $\motion$ is compact in the Euclidean topology, then, up to taking a subsequence, the limit $x = \lim_{t \to m^2} x_t$ exists in $\motion$.
  By continuity, $\distance(x) = m^2$.

  Let $\mathcal{U} \subset \motion$ be such that $\pi_{e'}(\mathcal{U})$ is a neighborhood of~$\pi_{e'}(x)$.
  Notice that the value of $s_{e'} = 0$ in $x$ since $\distance(x) = m^2$.
  Let $t \in \distance(\mathcal{U})$.
  Then there exist $y_1, y_2 \in \mathcal{U}$ such that $s_{e'}$ attains opposite sign in~$y_1$ and $y_2$ and $\distance(y_1) = \distance(y_2) = t$.
  Possibly shrinking $\mathcal{U}$, we can assume that $s_{\bar{e}}$ attains the same sign at $y_1$ and $y_2$.
  Since the map $\distance \colon \motion_{\bar{e}} \dashrightarrow \R$ is $2:1$, there exists a configuration $y_3 \in \motion$ such that $\distance(y_3) = t$,
  but $s_{\bar{e}}$ attains opposite signs at $y_2$ and $y_3$.
  This implies that $\distance \colon \motion \dashrightarrow \R$ has a fiber of cardinality $3$.

  Therefore, if $\Gal$ has cardinality $2$, then there must be a flat pose of $\Bipyramid$
  where the distance between~$\Top$ and~$\Bottom$ is~$m$.
  Arguing similarly, if $\Gal$ has cardinality $2$, then there must be a flat pose of $\Bipyramid$
  where the distance between~$\Top$ and~$\Bottom$ is~$M$.
\end{proof}

From now on we focus on pentagonal bipyramids, so we suppose that $\motion$ is a non-degenerate motion of~$\Penta$ with fixed edge lengths.

\begin{lemma}
\label{lemma:all_ones}
 The element $1^5 := (1,1,1,1,1)$ belongs to~$\Gal$.
\end{lemma}
\begin{proof}
Suppose that there is an element $\varphi\in\Gal$ with exactly three ones.
Let $e_1,e_2$ be the two equatorial edges such that $e_1, e_2 \notin \ones{\varphi}$.
By \Cref{lemma:element11}, there is $\psi\in\Gal$ such that $e_1,e_2 \in\ones{\psi}$.
If $\psi=1^5$, we are done.
Otherwise, there cannot be four nor three $1$'s in $\psi$
by \Cref{corollary:no_single_one_simultaneously,corollary:no_one}.
Hence, $\varphi+\psi=1^5$.

Therefore, we can assume that every element has either two or five $1$'s.
By \Cref{lemma:element11},
there are $\gamma_1,\gamma_2\in\Gal$
such that $\{(1,2),(2,3)\}\subseteq\ones{\gamma_1}$
and $\{(3,4),(4,5)\}\subseteq\ones{\gamma_2}$.
One of them has to be~$1^5$, otherwise $\gamma_1+\gamma_2$ has four $1$'s,
which contradicts \Cref{corollary:no_one}.
\end{proof}

In order to proceed, we recall the following known result, mentioned in \cite{Stachel1987, Nawratil2011}.

\begin{lemma}
\label{lemma:pyramid_flat_base}
	A pyramid with a quadrilateral base flexes so that the base stays coplanar if and only if
	the base is an antiparallelogram and the apex of the pyramid lies at the symmetry plane
	of the antiparallelogram.
\end{lemma}
\begin{proof}
 The result is known, we provide a computer algebra proof for self-containedness; see \cite[Section~1.1]{supportingCode}.
\end{proof}

\begin{lemma}
\label{lemma:planar_non_exists}
 There is no non-degenerate motion of a pentagonal bipyramid such that the realizations of the vertices~$1$, $3$, $\Top$, and $\Bottom$ stay coplanar during the flex.
 In particular, the element $(1,1,0,0,0)$ does not belong to $\Gal$.
\end{lemma}
\begin{proof}
 Assume that there exists such a non-degenerate motion whose general element has~$1$, $3$, $\Top$, and $\Bottom$ coplanar.
 By \Cref{lemma:pyramid_flat_base}, the points $1$, $3$, $\Top$, $\Bottom$ move as an antiparallelogram.

 Let $\Pyramid{2}$ denote the subgraph of $\Bipyramid$
 induced by the vertices $1,2,3,\Top$ and $\Bottom$.
 We construct a parametrized family of flexible pyramids
 that have the base vertices $1$, $3$, $\Top$, and~$\Bottom$ in common.
 This is achieved by selecting, for a realization of $1, 3, \Top, \Bottom$ in $\motion$ as an antiparallelogram and for any $s \in \N$, a different realization for $2$ that belongs to the symmetry plane of the antiparallelogram.
 Each of these realizations of~$\Pyramid{2}$ determines a motion~$\mathcal{N}_s$ of~$\Pyramid{2}$ so that the basis stays always coplanar.
 Hence, we can obtain a family $\{ \mathcal{N}_s \}_{s \in \N}$ such that $\mathcal{N}_0$ coincides with $\motion$ restricted to $\Pyramid{2}$ and for any $s \in \N$, the restriction of $\mathcal{N}_s$ to $\{1, 3, \Top, \Bottom\}$ coincides with the restriction of~$\motion$ to the same vertices.

 Now, from the motions~$\{ \mathcal{N}_s \}_{s \in \N}$, we determine motions~$\{ \motion_s \}_{s \in \N}$ for~$\Penta$ so that $\motion_0 = \motion$.
 To do so, for each $s \in \N$, we replace the realizations of vertex~$2$ in~$\motion$ by those of~$\mathcal{N}_s$.
 What we obtain are still motions of~$\Penta$ with given edge lengths, because what we are doing is to glue the realizations of the vertices $\{1,3,4,5,\Top,\Bottom\}$ with a pyramid with a vertex at the realization of $2$ along an antiparallelogram, and any pyramid with that antiparallelogram as a base fits.
 We denote by~$\lambda_{ab}^s$ the edge length of the edge $\{a,b\}$ in~$\motion_s$ for $s \in \N$.
 For any of these motions, we have $\lambda_{2\Top}^s = \lambda_{2\Bottom}^s$.

 For each of the motions~$\{ \motion_s \}_{s \in \N}$,
 the dihedral angle at the edge $\{2, \Top\}$ (or equivalently, the distance between $1$ and $3$) changes.
 By~\cite[Theorem~2]{Mikhalev2001} (see also~\cite{Gallet2022} for a generalization),
 there is hence an induced cycle~$(v_1, \dotsc, v_\ell, v_{\ell+1} = v_1)$ in~$\Penta$ containing $\{2, \Top\}$ such that we have $\sum_{i = 1}^{\ell} \varepsilon_i \lambda^s_{v_i, v_{i+1}} = 0$,  where $\varepsilon_i \in \{-1, +1\}$.
 However, there are only two possible induced cycles in the case of~$\Penta$, namely, $(\Top, 2, \Bottom, 4, \Top)$
 and $(\Top, 2, \Bottom, 5, \Top)$.
 The following part of the proof works analogously for either cycle, so from now on we suppose that the cycle is $(\Top, 2, \Bottom, 5, \Top)$. It follows that, for infinitely many $s \in \N$:
 \[
  \pm \lambda_{\Top 2}^s \pm \lambda_{2\Bottom}^s \pm \lambda_{\Bottom 5}^s \pm \lambda_{5 \Top}^s = 0 \,,
 \]
 (where $\pm$ here denotes \emph{some} choice between plus or minus, and not \emph{any} choice; the signs can be different for different $s$ but this does not influence the argument as there are only finitely many possibilities and infinitely many $s$').
 Since $\lambda_{\Top 2}^s = \lambda_{2\Bottom}^s$ for every $s \in \N$ and $\lambda_{\Bottom 5}^s$ and $\pm \lambda_{5 \Top}^s$ do not depend on~$s$, the only possibility is that $\lambda_{\Bottom 5}^s = \lambda_{5 \Top}^s$.

 It follows that the realizations of~$5$ belong to the symmetry plane of the segment between the realizations of~$\Top$ and~$\Bottom$.
 This plane is the same as the symmetry plane of the segment between the realizations of~$1$ and $3$.
 Therefore, the distance between the realizations of~$5$ and~$1$ is the same as the distance between the realizations of~$5$ and~$3$.
 The first distance, however, is constant during the motion~$\motion$, since $\{1, 5\}$ is an edge in~$\Penta$; thus, also the second distance is constant.

 This would imply that the realizations of~$3$, $4$, $5$, and~$\Top$ stay always at the same distance during the motion~$\motion$, which contradicts the assumption that the dihedral angle at the edge~$\{4, \Top\}$ changes during~$\motion$, since the motion is assumed to be non-degenerate.

 Therefore, it can never happen that $1$, $3$, $\Top$ and~$\Bottom$ stay coplanar during the motion.
 This implies by \Cref{lemma:consecutive_ones} that the element $(1,1,0,0,0)$ does not belong to~$\Gal$.
\end{proof}

\begin{proposition}
\label{proposition:galois_groups}
 Up to dihedral symmetry, the possibilities for~$\Gal$ are
 \[
  \begin{array}{l}
   \left\langle 1^5 \right\rangle \,, \\
   \left\langle 1^5, (1,0,1,0,0) \right\rangle \,.
  \end{array}
 \]
\end{proposition}
\begin{proof}
 From \Cref{lemma:all_ones} we know that $1^5 \in \Gal$.
 Hence, a first possibility for~$\Gal$ is to be the group $\left\langle 1^5 \right\rangle \cong \Z_2$.
 Suppose there is another non-trivial element in $\Gal$.
 By \Cref{corollary:no_one}, it has two or three $1$s.
 We focus on the elements with two ones, as each element with three $1$s
 can be obtained by summing up $1^5$ with an element with two $1$s.
 There cannot be an element with two neighboring $1$s by \Cref{lemma:planar_non_exists}
 (we count the first and last position being neighboring as well).
 Up to dihedral symmetry we can suppose that $\varphi = (1,0,1,0,0) \in \Gal$.
 By \Cref{corollary:no_single_one_simultaneously},
 the only other possible elements with two non-neighboring $1$'s are $(0,1,0,1,0)$ and $(0,1,0,0,1)$;
 however, these elements cannot be in~$\Gal$,
 because otherwise by adding any of them to~$\varphi$
 we would get an element with four~$1$s.
 This concludes the statement.
\end{proof}

To conclude the proof of our main result,
we need to show that we can construct instances of flexible pentagonal bipyramids
whose Galois groups match each of the two possibilities from \Cref{proposition:galois_groups}.

\begin{theorem}
\label{theorem:nelson_Z2}
 There exists a flexible pentagonal bipyramid obtained by the Nelson construction
 from two flexible octahedra of Type~III.
 This motion has Galois group~$\left\langle 1^5 \right\rangle$.
\end{theorem}
\begin{proof}
 We exhibit an example that has been created via the technique of \emph{motion polynomials}.
 Explaining how this was done goes beyond the scope of this paper and will be the subject of a further one.
 For the time being, we limit ourselves to providing an explicit parametrization of the motion;
 from the parametrization, it is possible to check that the flexible pentagonal bipyramid comes
 from a Nelson construction applied to two Type III flexible octahedra;
 see \cite[Section~1.2]{supportingCode}.
 The parametrization is the following:
 \begin{align*}
  \rho_t(\Top) &= \left(0, \frac{1}{4}, 0\right), \qquad
  \rho_t(2) = \left(\frac{1}{6}, 0, 0\right), \qquad
  \rho_t(3) = \left(0, 0, 0\right),\\
  \rho_t(1) &= \frac{3}{28\,(9\,t^2 + 13)}\,\left(-9\,t^2 + 10, \frac{69\,t^2 + 161}{4}, 23\,t\right),\\
  \rho_t(4) &= \frac{989}{493\,\left(t^2 + 4) \, (4\,t^2 + 1\right)}\,\left(\frac{-4\,t^4 + t^2 - 4}{6}, 0, 2\,t^3 + 2\,t\right),\\
  \rho_t(5) &= \frac{69}{9794\,(9\,t^2 + 13)}\,\left(\frac{126\,t^4 - 2429\,t^2 - 1244}{t^2 + 4},
  \frac{231\,t^2 + 2971}{4}, \frac{-1063\,t^3 + 992\,t}{t^2 + 4}\right),\\
  \rho_t(\Bottom) &= \frac{1}{9}\,\left(-\frac{4}{3}, \frac{-t^2 + 1}{t^2 + 1}, \frac{2\,t}{t^2 + 1}\right)\,.
 \end{align*}
 One may notice that the triangle $\{2,3,\Top\}$ is fixed throughout the motion.
 There are two flat poses: all vertices are in the $xy$-plane for $t=0$ and infinity.
 The position of the removed vertex 0 in the Nelson construction is
 \[
  \rho_t(0) = \frac{23}{157}\,\left(\frac{-3\,t^2 + 12}{2\,t^2 + 8}, 2, \frac{6\,t}{t^2 + 4}\right)\,.
 \]
 The motion is illustrated in \Cref{fig:flexNelsonIII}.

 The Galois group of the motion has to be $\langle 1^5 \rangle$,
 otherwise there is a pair of equatorial edges with the same edge length
 by \Cref{corollary:edge_lengths_two_ones}:
 a direct symbolic computation shows that this is not the case;
 see \cite[Section~1.2]{supportingCode}.
\end{proof}

\begin{figure}[ht]
	\centering
	\begin{tikzpicture}
		\clip (1.75,-1.6) rectangle (15.7,-11.4);
		\foreach \i in {1,2,...,4}
		{
			\foreach \j [evaluate=\j as \k using int((\j-1)*4+\i)] in {1,2,3}
			{
				\node[] at (3.5*\i,-3.25*\j) {\includegraphics[width=3.5cm,trim=2.5cm 3cm 2.5cm 1.8cm,clip=true]{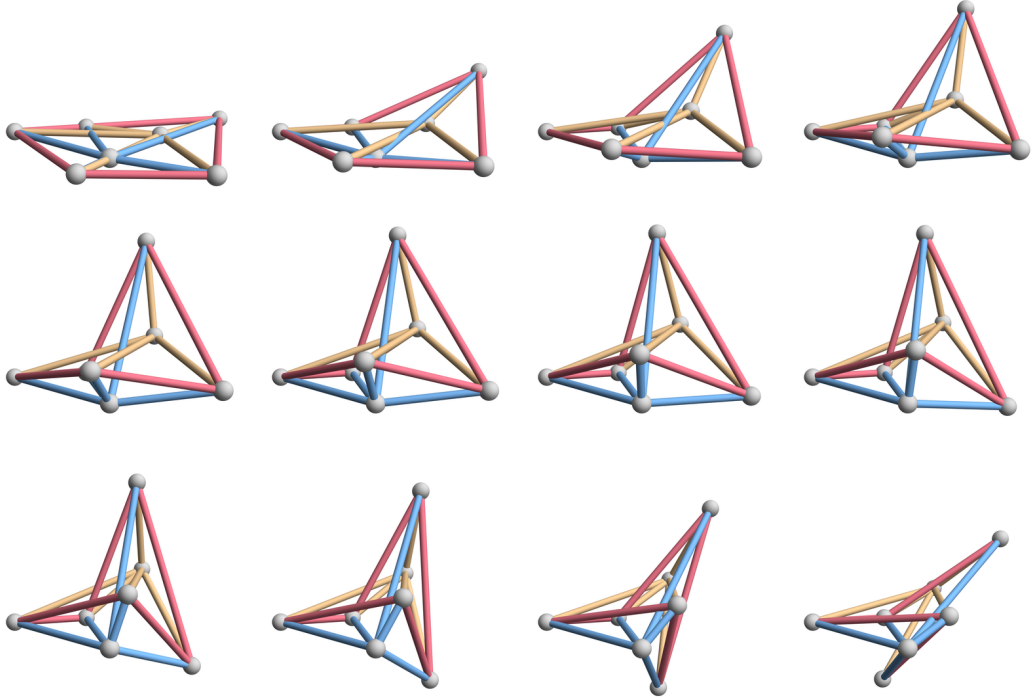}};
			}
		}
	\end{tikzpicture}
	\caption{Instances of the motion of the pentagonal bipyramid based on Nelson construction of two flexible octahedra with Type~III.}
	\label{fig:flexNelsonIII}
\end{figure}

\begin{theorem}
\label{theorem:nelson_opposite}
 There exists a $5$-dimensional family of flexible pentagonal bipyramids without a flat pose obtained
 by the Nelson construction from two flexible octahedra that are Type~I.
 These motions have Galois group $\left\langle 1^5, (1,0,1,0,0) \right\rangle$.
\end{theorem}
\begin{proof}
 We show that such a Nelson construction from two Type~I flexible octahedra
 can be performed.
 The line-symmetry makes this case less involved than the previous one.
 To fix notations, let us suppose that $\Oct_{\kasperl}$ has
 vertices $\{0,1,2,3,\Top,\Bottom\}$ and non-edges $\{0,2\}$, $\{1,3\}$, and~$\{\Top, \Bottom\}$,
 while $\Oct_{\pezi}$ has vertices $\{0,3,4,5,\Top,\Bottom\}$ and non-edges $\{0,4\}$, $\{3,5\}$, and~$\{\Top, \Bottom\}$.
 We suppose that the Nelson construction is done so that the vertices~$0$, $1$, and~$5$ stay collinear
 during the flex.
 We start by constructing a fitting pair of Type~I flexible octahedra.
 \begin{itemize}
  \item
  Pick two lines $\ell_1$ and $\ell_2$ in~$\R^3$ intersecting in a point~$M$.
  \item
  Pick the realizations of~$\Top$ and $\Bottom$ so that they are on the common normal of~$\ell_1$ and~$\ell_2$ and their midpoint is~$M$.
  \item
  Put the realization of vertex~$3$ anywhere in~$\R^3$, except on the plane spanned by~$\ell_1$ and~$\ell_2$ and
  on the two planes that bisect the two angles determined by~$\ell_1$ and~$\ell_2$.
  \item
  Let the realizations of~$1$ and~$5$ be the reflections of~$3$ across~$\ell_1$ and~$\ell_2$, respectively.
  \item
  Put the realization of~$0$ anywhere on the line $15$.
  \item
  Define the realizations of~$2$ and~$4$ to be the reflections of~$0$ across~$\ell_1$ and~$\ell_2$, respectively.
 \end{itemize}
 By construction, the two realizations of~$\Oct_{\kasperl}$ and~$\Oct_{\pezi}$ determine a fitting pair and each of them allows a flex that provides a Type~I motion (see \Cref{basics:octahedra}).
 The question is whether these two Type~I motions actually determine a flex of the fitting pair.
 To prove this, notice that the previous construction depends on~$12$ real parameters:
 \begin{itemize}
  \item $7$ parameters for the two lines~$\ell_1$ and~$\ell_2$;
  \item $1$ parameter for the realizations of~$\Top$ and~$\Bottom$;
  \item $3$ parameters for the realization of~$3$;
  \item $0$ parameters for the realizations of~$1$ and~$5$ (i.e., they are uniquely determined);
  \item $1$ parameter for the realization of~$0$;
  \item $0$ parameters for the realizations of~$2$ and~$4$ (i.e., they are uniquely determined).
 \end{itemize}
 Hence, the locus of possible fitting pairs from the previous construction is diffeomorphic to an open subset~$\mathcal{U}$ of~$\R^{12}$.
 Given a fitting pair, we measure the squared distances of the edges
 \[
  \{0, 1\} \,, \{0, 5\} \,, \{0, \Top\} \,, \{1, \Top\} \,, \{2, \Top\} \,, \{3, \Top\} \,, \{1, 2\} \,.
 \]
 This yields a differentiable map $f \colon \mathcal{U} \rightarrow \R^7$.
 By the line symmetry of $\Oct_{\kasperl}$ and~$\Oct_{\pezi}$,
 the edge length of any of the other edge equals the edge length of some edge in the list.

 We claim that, if we take coordinates $\omega_{01}$, $\omega_{05}$, $\omega_{0\Top}$, $\omega_{1\Top}$, $\omega_{2\Top}$, $\omega_{3\Top}$, $\omega_{12}$ on~$\R^7$, then the image of~$f$ is contained in the zero set of the polynomials
 \[
  \omega_{1\Top} + \omega_{2\Top} - \omega_{0\Top} - \omega_{3\Top}
  \quad \text{and} \quad
  \omega_{01} \, \omega_{05} - (\omega_{0\Top} - \omega_{1\Top})^2 \,.
 \]
 The equation $\omega_{01} \, \omega_{05} = (\omega_{0\Top} - \omega_{1\Top})^2$ is a consequence of Pythagoras' theorem.
 In fact, let $\rho$ be a realization obtained by the previous procedure and $\lambda$ be the corresponding induced edge lengths.
 Consider the isosceles triangle given by the realizations of~$1$, $5$, and~$\Top$.
 Let $F$ be the foot of the height from~$\Top$ and set
 \[
  h := \left\| \rho(\Top) - \rho(F) \right\| \,, \;
  u := \left\| \rho(1) - \rho(F) \right\| = \left\| \rho(5) - \rho(F) \right\|
  \; \text{and} \;
  v := \left\| \rho(0) - \rho(F) \right\| \,.
 \]
 \begin{center}
  \begin{tikzpicture}
   \node[vertex, label={[labelsty,label distance=-2pt]90:$\Top$}] (T) at (0, 3) {};
   \node[vertex, label={[labelsty,label distance=-2pt]270:$1$}] (1) at (-3,0) {};
   \node[vertex, label={[labelsty,label distance=-2pt]270:$0$}] (0) at (-2,0) {};
   \node[vertex, label={[labelsty,label distance=-2pt]270:$5$}] (5) at (3,0) {};
   \node[vertex, label={[labelsty,label distance=-2pt]270:$F$}] (F) at (0,0) {};
   \draw[edge] (1)--(0);
   \draw[edge] (0) to node[labelsty,below] {$v$} (F);
   \draw[edge] (F) to node[labelsty,below] {$u$} (5);
   \draw[edge] (T) to node[labelsty,left] {$\lambda_{1 \Top}$} (1);
   \draw[edge] (T) to node[labelsty,near end, below right] {$\lambda_{0 \Top}$} (0);
   \draw[edge] (T) to node[labelsty,right] {$\lambda_{1\Top}$} (5);
   \draw[edge] (T) to node[labelsty,right] {$h$} (F);
  \end{tikzpicture}
 \end{center}
 Then
 \[
   \lambda_{1\Top}^2 = u^2 + h^2
   \quad \text{and} \quad
   \lambda_{0\Top}^2 = v^2 + h^2 \,,
 \]
 thus
 \[
   \lambda_{1\Top}^2 - \lambda_{0\Top}^2 =
  (u + v)(u - v) =
  \lambda_{01} \, \lambda_{15} \,.
 \]
 Squaring both sides implies the relation about the variables~$\omega$.
 The second equation $\omega_{1\Top} + \omega_{2\Top} = \omega_{0\Top} + \omega_{3\Top}$ can be derived as follows. Take coordinates so that the two lines~$\ell_1$ and $\ell_2$ lie on the $\{z = 0\}$ plane in~$\R^3$ and their intersection point~$M$ is~$(0,0,0)$.
 Then by performing the previous construction, one finds that the coordinates of the points in the realization are of the form
 \begin{align*}
  \rho(\Top) &= (0, 0, a) \,, & \rho(\Bottom) &= (0, 0, -a) \,, & \rho(3) &= (x_3, y_3, z_3) \,, \\
  \rho(1) &= (x_3, -y_3, -z_3) \,, & \rho(0) &= (x_0, y_0, -z_3) \,, & \rho(2) &= (x_0, -y_0, z_3) \,.
 \end{align*}
 A direct computation then shows that the equation in the $\omega$ variables holds.
 This concludes the proof of the claim.

 The claim implies that the image of~$f$ is at most $5$-dimensional.
 Then Sard's theorem (see for example \cite[Theorem~1.30]{Aubin2001}) guarantees that the fiber of~$f$ over a general point of its image is a submanifold of~$\mathcal{U}$ of dimension at least~$12 - 5 = 7$.
 Now, the group of direct isometries acts on the elements of~$\mathcal{U}$ and it is $6$-dimensional, which means that there exists at least a $1$-dimensional submanifold of realizations of fitting pairs in a general fiber of~$f$ whose elements are not pairwise congruent.
 This means that fitting pairs in such a fiber determine a flex.

 It remains to prove that the Galois group is $\left\langle 1^5, (1,0,1,0,0) \right\rangle$.
 Suppose that there is a flat pose.
 The vertices $\Top$ and $\Bottom$ cannot coincide,
 since $\lambda_{3\Top} \neq \lambda_{3\Bottom}$ by the construction.
 Hence, in this pose the symmetry lines of $\Oct_{\kasperl}$ and~$\Oct_{\pezi}$ coincide,
 which forces the vertices $1$ and $5$ to coincide as well, as they are both reflections of $3$.
 But this is again not possible as $\lambda_{13} \neq \lambda_{35}$
 by the assumption on the position of vertex $3$ in the construction.
 Therefore, there is no flat pose,
 so the Galois group is not $\langle 1^5 \rangle$ by \Cref{proposition:Z2_implies_flat poses}
 and the statement follows by \Cref{proposition:galois_groups}.
\end{proof}
An example of a motion constructed as in the proof is illustrated in \Cref{fig:flexNelsonI}.

\begin{figure}[ht]
	\centering
	\begin{tikzpicture}
		\foreach \i in {1,2,...,4}
		{
			\foreach \j [evaluate=\j as \k using int((\j-1)*4+\i)] in {1,2,3}
			{
				\node[] at (3.5*\i,-3*\j) {\includegraphics[width=3.5cm,trim=1.7cm 2cm 1.7cm 2cm,clip=true]{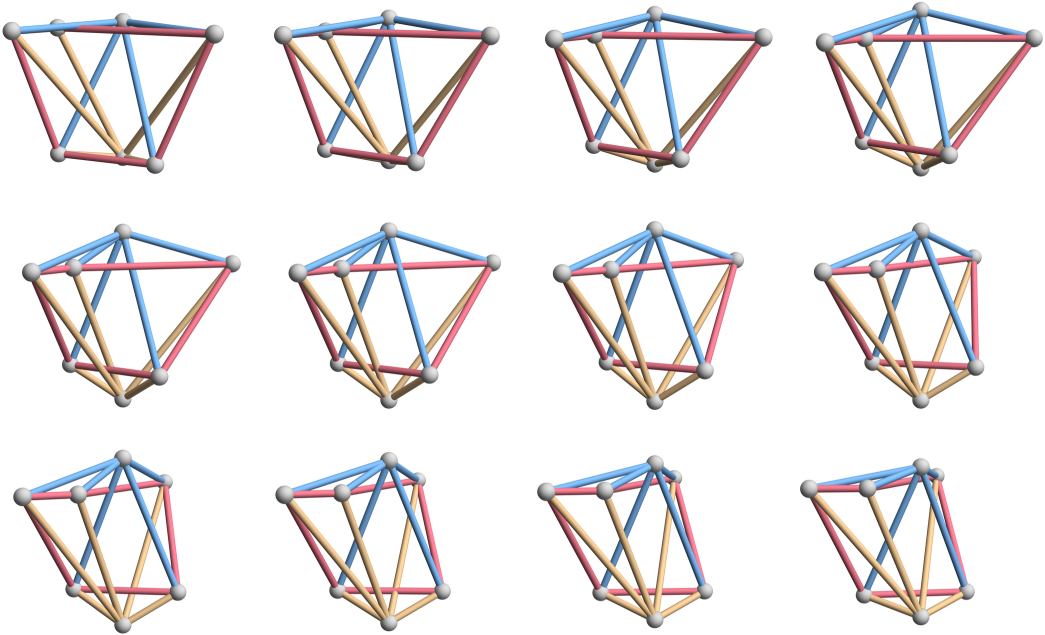}};
			}
		}
	\end{tikzpicture}
	\caption{Instances of the motion of the pentagonal bipyramid based on Nelson construction of two flexible octahedra with Type~I.}
	\label{fig:flexNelsonI}
\end{figure}

\section{A flexible embedded polyhedron with \texorpdfstring{$8$}{8} vertices}
\label{example}

A pleasant consequence of the Nelson construction presented in the proof
of \Cref{theorem:nelson_opposite} is the possibility to construct
a flexible embedded polyhedron with $8$ vertices.
As recalled in the introduction, Maksimov \cite{Maksimov2008} showed that
the only possibility for the combinatorial structure of
a flexible embedded polyhedron with fewer than $9$ vertices
(i.e., with fewer vertices than Steffen's polyhedron)
is the one of a pentagonal bipyramid with one of its faces subdivided into three
by adding a new vertex~$\New$ (see \Cref{figure:structure_8_vertices});
we call this the \emph{subdivided pentagonal bipyramid},
denoted by~$\Subdivided$.

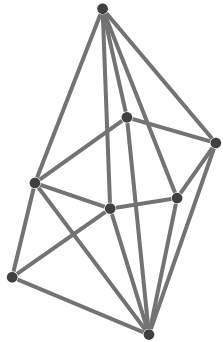
\begin{figure}[ht]
  \centering
  \begin{tikzpicture}
    \node[vertex] (T) at (0.89, 4.56) {};
    \node[vertex] (B) at (1.50,0.24) {};
    \node[vertex] (1) at (0.99,1.91) {};
    \node[vertex] (2) at (0,2.25) {};
    \node[vertex] (3) at (1.21,3.12) {};
    \node[vertex] (4) at (2.38,2.78) {};
    \node[vertex] (5) at (1.87,2.05) {};
    \node[vertex] (N) at (-0.3,1) {};
    \draw[edge] (T)--(1);
    \draw[edge] (B)--(1);
    \draw[edge] (T)--(2);
    \draw[edge] (B)--(2);
    \draw[edge] (T)--(3);
    \draw[edge] (B)--(3);
    \draw[edge] (T)--(4);
    \draw[edge] (B)--(4);
    \draw[edge] (T)--(5);
    \draw[edge] (B)--(5);
    \draw[edge] (1)--(2);
    \draw[edge] (2)--(3);
    \draw[edge] (3)--(4);
    \draw[edge] (4)--(5);
    \draw[edge] (5)--(1);
    \draw[edge] (B)--(N);
    \draw[edge] (1)--(N);
    \draw[edge] (2)--(N);
  \end{tikzpicture}
  \caption{The graph of subdivided pentagonal bipyramid~$\Subdivided$.}
  \label{figure:structure_8_vertices}
\end{figure}

Notice that, since the new vertex~$\New$ in~$\Subdivided$ is connected to three vertices
of a pentagonal bipyramid~$\Penta$, then any realization of~$\Subdivided$
with non-coplanar faces admits exactly one other realization with the same edge lengths
and coinciding on the vertices of~$\Penta$.

So far, we considered $\Penta$ and $\Subdivided$ as graphs, since we only took care of vertices and edges.
We now want to think of them as polyhedra, namely considering their faces realized as triangles.
In particular, we introduce the notions to formalize the concept of self-intersections of faces.
To do so, we consider polyhedra as instances of the more general notion of realizations of abstract simplicial complexes.
We introduce for the latter the concept of embeddedness, which describes the requirement that faces do not intersect unless they share an edge.

\begin{definition}
 Let $V$ be a finite set.
 An \emph{abstract simplicial complex} is a subset~$\mathcal{K}$ of~$2^V$ that is closed under taking subsets.
 The elements of~$V$ are called the \emph{vertices} of~$\mathcal{K}$.
 The elements $\sigma \in \mathcal{K}$ are called \emph{simplices}.
 The simplices with two elements are called \emph{edges}.
\end{definition}

\begin{definition}
 The \emph{geometric realization}~$|\mathcal{K}|$ of a simplicial complex~$\mathcal{K}$ is the subset of~$\R^V$ given by the union of the convex hulls of the sets $\{e_v\}_{v \in \sigma}$ for $\sigma \in \mathcal{K}$, where the elements~$\{e_v\}_{v \in V}$ are the standard basis of~$\R^V$.
 A \emph{realization} of~$\mathcal{K}$ is a function $\rho \colon |\mathcal{K}| \rightarrow \R^3$ that is affine-linear on each $|\sigma|$ for $\sigma \in \mathcal{K}$.
 Notice that a realization~$\rho$ is completely determined by the images of the vertices of~$\mathcal{K}$.
\end{definition}

\begin{definition}
 A \emph{polyhedron} is a realization~$\rho$ of a simplicial complex~$\mathcal{K}$ such that $|\mathcal{K}|$ is homeomorphic to a $2$-sphere.
 We say that a realization $\rho$ yields an \emph{embedded polyhedron}
 or a \emph{polyhedron with no self-intersections}
 if the function~$\rho$ is injective, i.e., if $\rho$ is a homeomorphism onto its image.
\end{definition}
Notice that $\Bipyramid$ in \Cref{definition:flexible} can be replaced by any graph,
which allow us to state the following definition.
\begin{definition}
 A \emph{flex} of a polyhedron~$\rho \colon \mathcal{K} \rightarrow \R^3$
 is a continuous function $f \colon [0,1) \longrightarrow (\R^3)^V$ such that
 $f$ is a flex of~$\rho$ considered as a realization of the graph given
 by the vertices and edges of~$\mathcal{K}$.
 We say that a flex~$f$ yields a \emph{flex of an embedded polyhedron}
 if $f(t)$ is an embedded polyhedron for all $t \in [0,1)$.
\end{definition}

From now on, we think of~$\Penta$ as the abstract simplicial complex with simplices
\begin{gather*}
 \{\Top, 1, 2\}, \; \{\Top, 2, 3\}, \; \{\Top, 3, 4\}, \; \{\Top, 4, 5\}, \; \{\Top, 5, 1\}, \\
 \{\Bottom, 1, 2\}, \; \{\Bottom, 2, 3\}, \; \{\Bottom, 3, 4\}, \; \{\Bottom, 4, 5\}, \; \{\Bottom, 5, 1\},
\end{gather*}
and their subsets.
Similarly, the simplicial complex $\Subdivided$ is obtained by removing $\{\Bottom, 1, 2\}$ from $\Penta$ and by adding $\{\New, \Bottom, 1\}$, $\{\New, 1, 2\}$, and $\{\New, \Bottom, 2\}$, together with their subsets.

We know by \cite[Theorem~2]{Connelly1978} that all pentagonal flexible bipyramids must have self-intersections, namely that for each realization in a flex there are at least two triangular faces that intersect in more than their possible common edge.
Hence, the only possibility to obtain a flexible realization~$\Subdivided$
with no self-intersection is that we can produce a flex of a realization of~$\Penta$
for which all the possible intersections between faces always involve a common one;
if we are in this situation,
then we can hope that removing the latter face by subdividing it into three faces
yields a flex for~$\Subdivided$ that avoids the previous self-intersections
and does not create any other one, see \Cref{fig:avoidable_self_intersection}.

\begin{figure}[ht]
  \centering
  \includegraphics[width=5cm,clip=true,trim= 0.2cm 0.5cm 0.2cm 0.1cm]{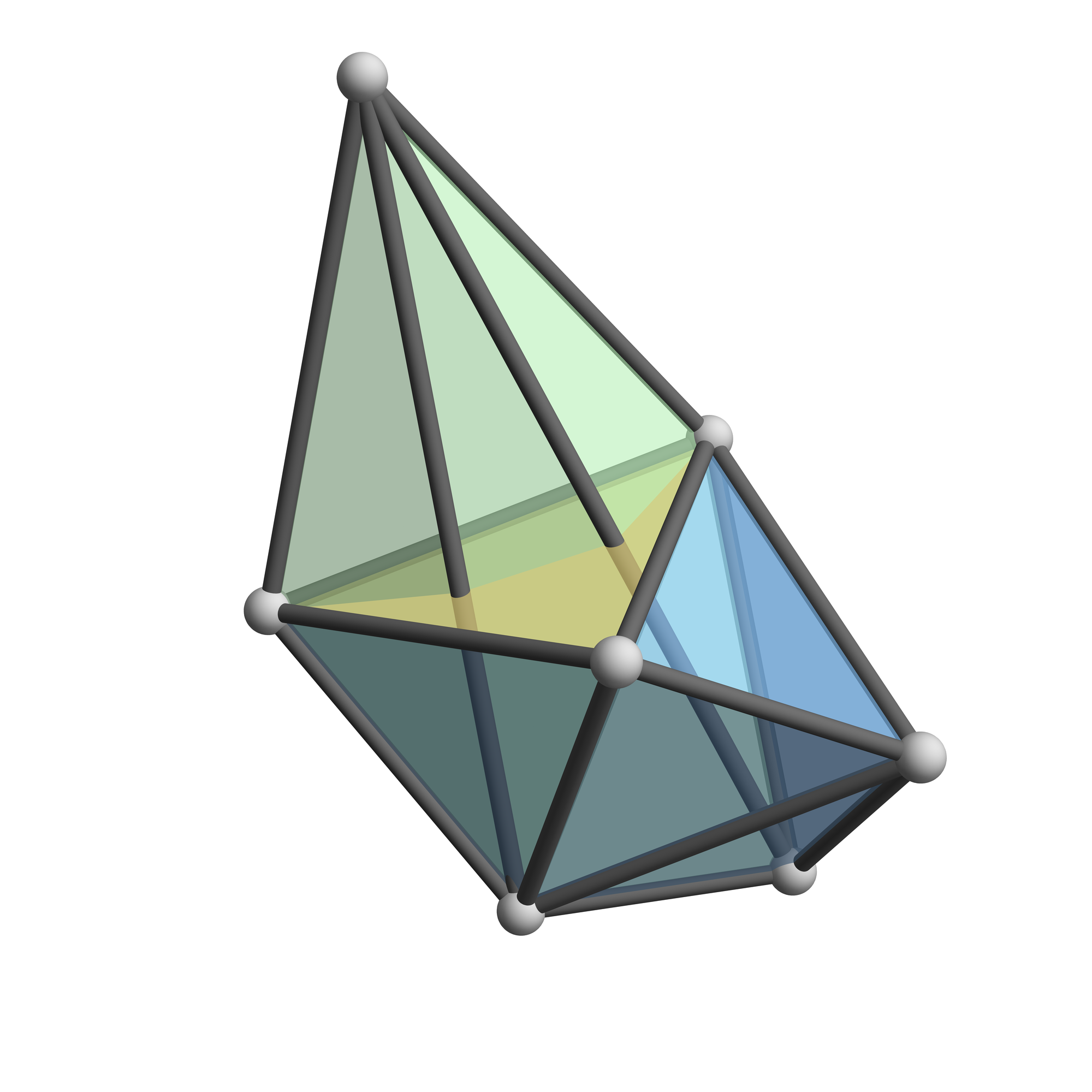}
  \caption{An illustration of a self-intersection that can be avoided by subdividing the yellow face
  into three new faces so that the new vertex is ``above'' the current top one.}
  \label{fig:avoidable_self_intersection}
\end{figure}

Let us revisit the Nelson construction discussed in the proof of \Cref{theorem:nelson_opposite}.
Let $p_0, \dotsc, p_5, p_{\Top}, p_{\Bottom}$ be the realizations of the vertices of~$\Oct_\kasperl$ and~$\Oct_\pezi$ according to the construction.
Using rotations and translation, we can suppose that the line $\ell_1$ is the $x$-axis and that the intersection point~$M$ is the origin~$(0,0,0)$.
Moreover, we can suppose that line $\ell_2$ is such that the reflection w.r.t.\ it
corresponds to multiplication by the matrix
\[
 \begin{pmatrix}
  c & s & 0\\
  s & -c & 0 \\
  0 & 0 & -1
 \end{pmatrix}\,,
\]
where $c = \frac{t^2 -1}{t^2 + 1}$ and $s = \frac{2t}{t^2 + 1}$ for some real number $t \in \R$.
In other words, $\ell_2$ is the line in the $xy$-plane forming an angle $\arccos(c)/2$ with the $x$-axis.
Moreover, we can write $p_{\Top} = (0, 0, a)$ and $p_{\Bottom} = (0, 0, -a)$ for some real number $a \in \R$.
If we write $p_3 = (x_3, y_3, z_3)$, then
\[
 p_1 = (x_3,\,-y_3,\,-z_3)
 \quad \text{and} \quad
 p_5 = (c \, x_3 + s \, y_3,\, -c \, y_3 + s \, x_3,\, -z_3) \,.
\]
Finally, we denote by~$\mu$ the real number such that $p_0 = \mu \, p_1 + (1 - \mu) \, p_5$.
Then
\begin{align*}
 p_2 &= \bigl(
  -(\mu (c - 1) - c) x_3 - (\mu s - s) y_3,\,
  (\mu s - s) x_3 - (\mu (c - 1) - c) y_3,\,
  z_3
 \bigr) \,,\\
 p_4 &= \bigl(
  -\mu s \, y_3 + (\mu (c - 1) + 1) x_3,\,
  \mu s \, x_3 + (\mu (c - 1) + 1) y_3,\,
  z_3
 \bigr) \,.
\end{align*}

In conclusion, we can write the coordinates of all the realizations of the vertices of~$\Oct_\kasperl$ and~$\Oct_\pezi$ as rational functions in the six parameters
\[
 (t, a, x_3, y_3, z_3, \mu) \,.
\]
While in general these parameters give polyhedra with many self intersections,
a vast but educated computer search finally allowed us to find a set of parameters for which there are only three intersections all of which contain one particular face.
There are of course infinitely many such parameters. We present here one set for which the intersections are visible.

\begin{lemma}
\label{lemma:nelson_few_intersections}
 The realization of~$\Penta$ given by the Nelson construction with the choice of parameters
 \[
  (t, a, x_3, y_3, z_3, \mu) = (-5/8, \, 1, \, 15/7, \, 11/4, \, 5/2, \, 2/7)
 \]
 has only three self-intersections, namely, the ones between the faces
 \[
  \{\Bottom, 1, 2\} \text{ and } \{\Top, 4, 5\} \,,
  \quad
  \{\Bottom, 1, 2\} \text{ and } \{\Top, 5, 1\} \,,
  \quad
  \{\Bottom, 1, 2\} \text{ and } \{\Bottom, 4, 5\} \,.
 \]
\end{lemma}
\begin{proof}
 The proof follows from a direct symbolic computation,
 checking all possible intersections between pairs of faces of~$\Penta$.
 We provide such a symbolic check in \cite[Section~2.2]{supportingCode}.
 The intersections are displayed in \Cref{fig:threeIntersections}.
\end{proof}

\begin{figure}[ht]
	\centering
	\includegraphics[trim=3cm 0cm 2cm 0cm,clip=true,scale=0.5]{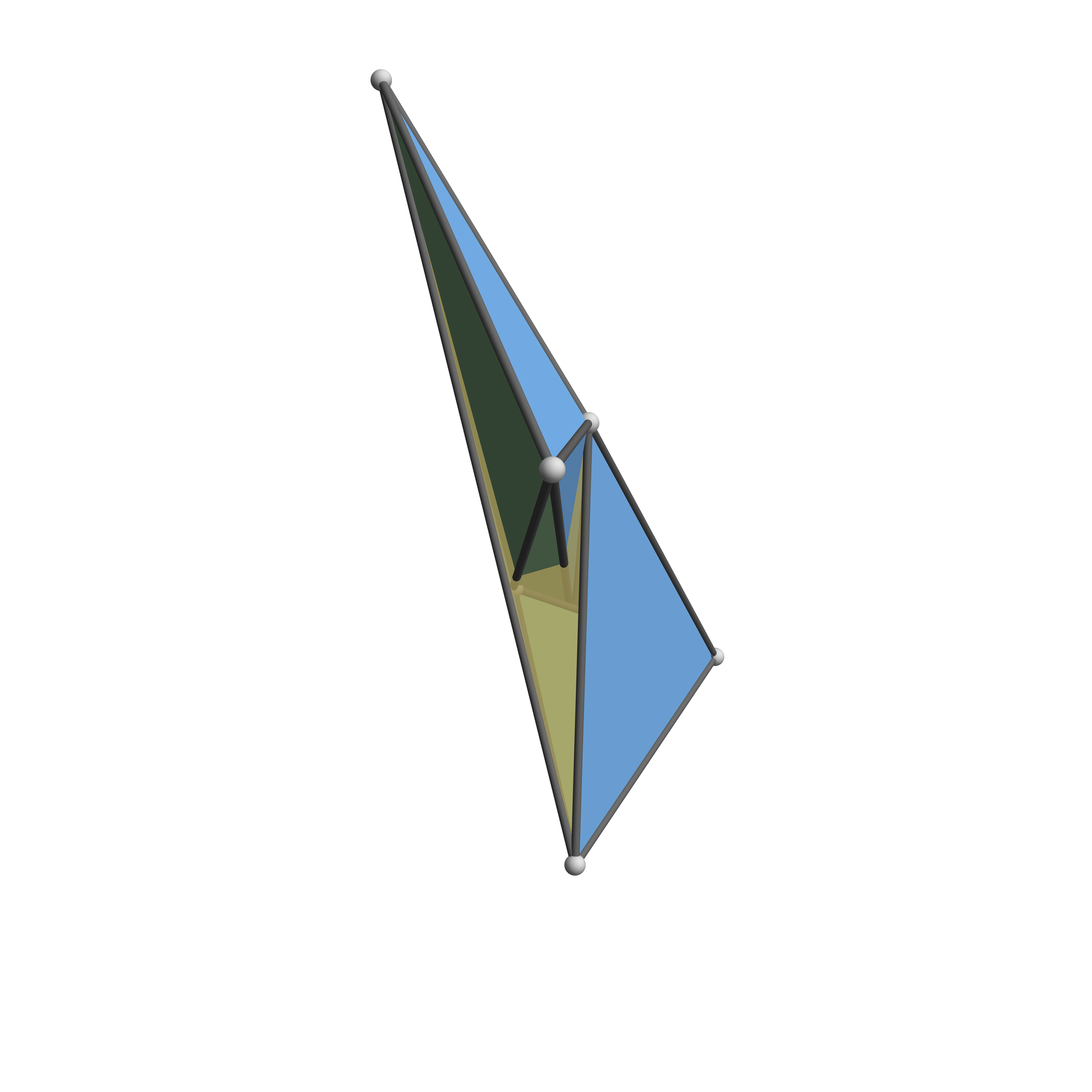}
	\includegraphics[trim=0cm 0cm 1cm 0cm,clip=true,scale=0.5]{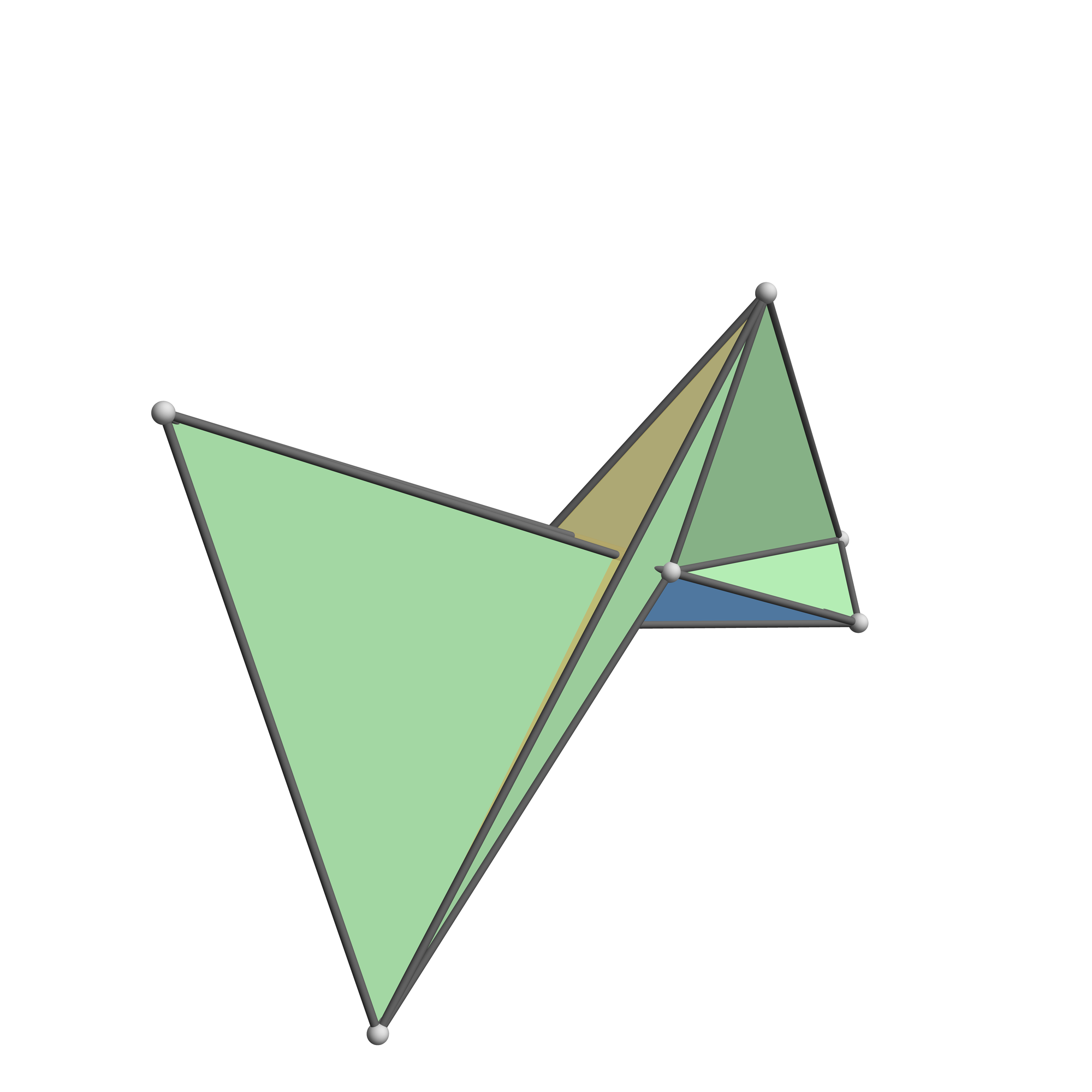}
	\includegraphics[scale=0.5]{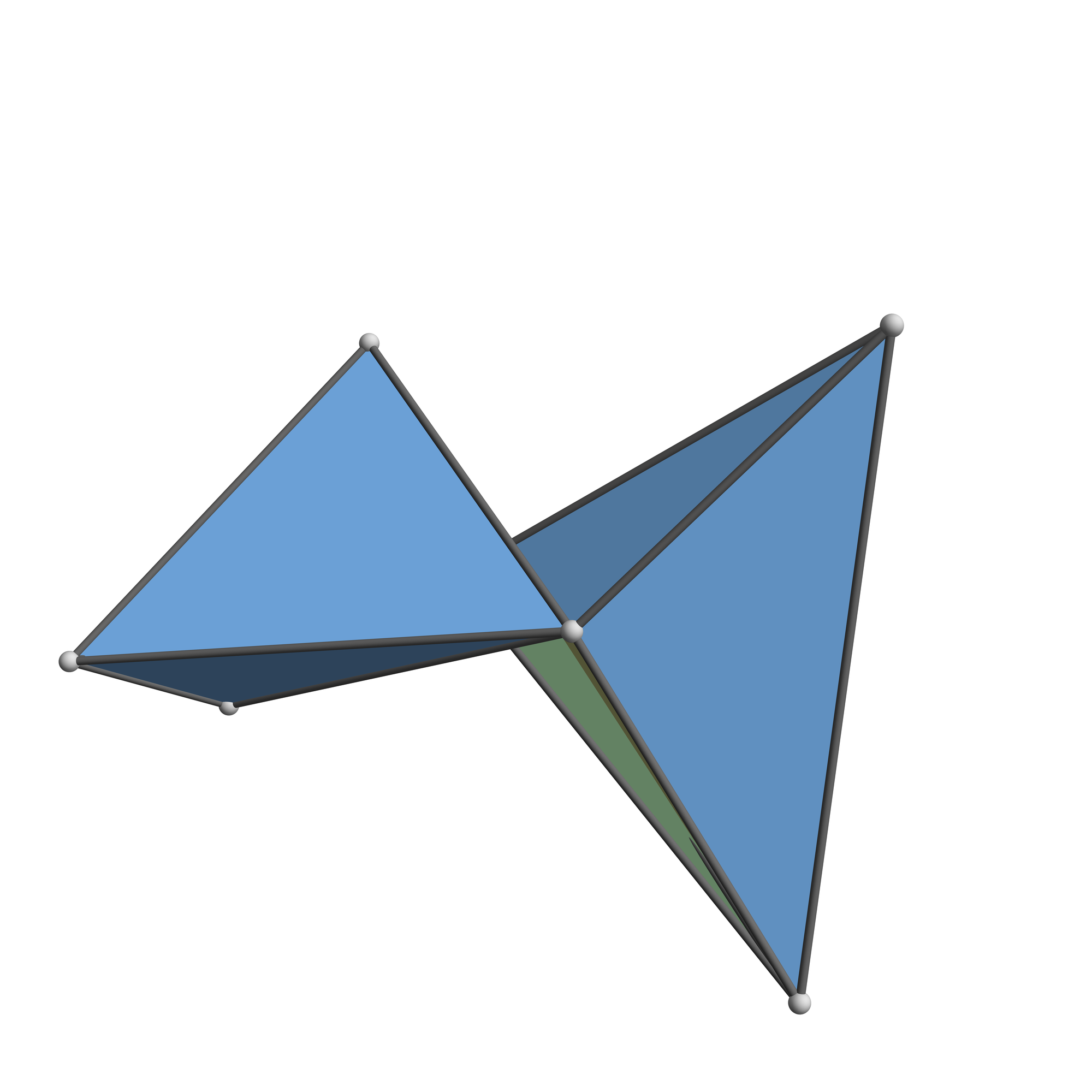}
	\includegraphics[trim=2cm 0cm 3cm 0cm,clip=true,scale=0.5]{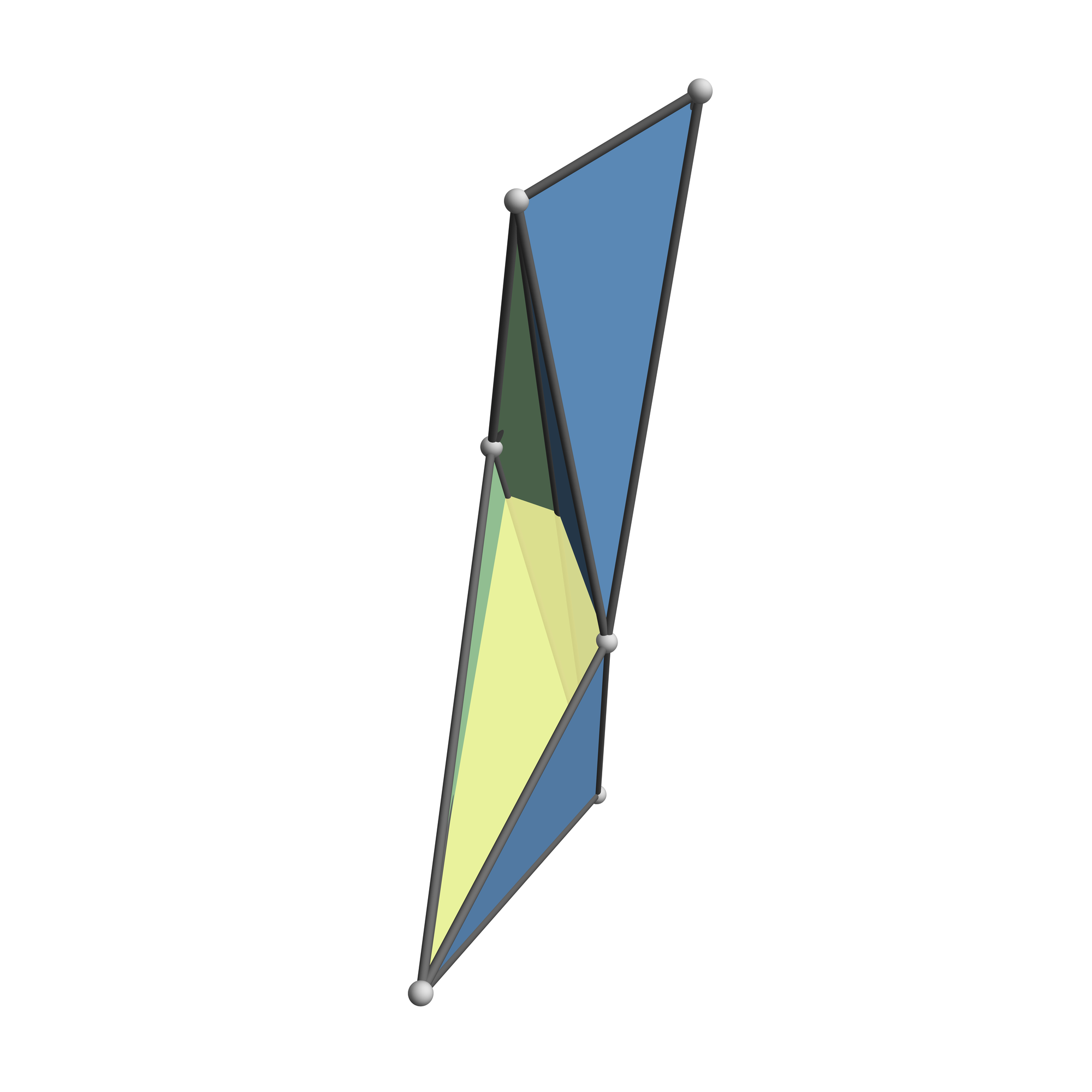}
	\caption{Four views of the pentagonal bipyramid showing the three self intersections.}
	\label{fig:threeIntersections}
\end{figure}

\begin{theorem}
\label{theorem:existence_flexible}
 The realization of~$\Subdivided$ obtained by mapping the vertices of~$\Penta$
 as described in \Cref{lemma:nelson_few_intersections} and the vertex~$\New$
 to $p_1 - \frac{9}{10}(p_2-p_1) + 3(p_\Bottom-p_1) + \frac{3}{5}(p_2-p_1)\times(p_\Bottom-p_1)$
 has a flex with no self-intersections,
 hence yields an example of a flexible embedded polyhedron with $8$ vertices.
\end{theorem}
\begin{proof}
 Let $\rho$ be the realization of~$\Subdivided$ described in the statement.
 As in \Cref{lemma:nelson_few_intersections}, the fact that $\rho$ is free from self-intersections can be checked via symbolic computation.
 We provide such a symbolic check in \cite[Section~2.3]{supportingCode}.

 We consider the locus~$\mathcal{X}$ of parameters $(t, a, x_3, y_3, z_3, \mu)\in\R^6$
 whose corresponding realizations induce the same edge lengths as $\rho$.
 We prove that $\rho$ admits a flex by exhibiting a parametrization of a real curve in~$\mathcal{X}$
 passing through the parameters specifying $\rho$, i.e., those in \Cref{lemma:nelson_few_intersections}.
 Hereafter, we report such a parametrization in terms of the variable parameter~$a$;
 the value corresponding to~$\rho$ is $1$:
 \begin{align*}
  t &= -\frac{\sqrt{9529} \sqrt{A_2}}{76232 \, a} \,, & z_3 &= \frac{5}{2 a} \,, \\
  x_3 &= \frac{\sqrt{89}}{623} \,  \sqrt{\frac{C_{1} - 20 \, \sqrt{A_{1} A_2}}{B}} \,, &  m &= \frac2{7} \,, \\
  y_3 &= \frac{\sqrt{89}}{2492 \, a} \sqrt{\frac{320 a^2 \, \sqrt{A_{1} A_2} - C_2}{B}} \,,
 \end{align*}
 where
 \begin{align*}
  A_1 &= -2029832431225781 \, a^4 + 3583108625879406 \, a^2 - 1376807431326600 \,, \\
  A_2 &= -69776 \, a^4 + 744101 \, a^2 - 436100 \,,\\
  B &= -3419024 \, a^4 + 41949653 \, a^2 - 21368900 \,,\\
  C_1 &= 14910363664 \, a^{6} - 192762292317 \, a^4 + 1211236673778 \, a^2 - 560045725400 \,, \\
  C_2 &= 4472093788624 \, a^{6} - 40400312279273 \, a^4 + 38266083804200 \, a^2 - 9318977290000 \,.
 \end{align*}
 We provide the derivation of such a parametrization starting from the defining equations of~$\mathcal{X}$
 together with an independent check of its correctness in \cite[Sections~2.4--2.6]{supportingCode}, where we also report an alternative proof for the flexibility.
\end{proof}

A final note: the example of flexible embedded polyhedron we have constructed exhibits an extremely limited mobility; it is possible that the methods of \cite{Lijingjiao2015} could be applied to obtain an instance with a wider range of mobility.

\section*{Acknowledgments}

We thank Zijia Li for suggesting us the use of motion polynomials for the construction we provide in \Cref{theorem:nelson_Z2}.
Jan Legerský is grateful for the discussions with Robert Connelly
during the \emph{Fields Institute Focus Program on Geometric Constraint Systems} in Toronto in 2023.
Part of this project was accomplished during the special semester on \emph{Rigidity and Flexibility} at the \emph{Johann Radon Institute for Computational and Applied Mathematics (RICAM) of the Austrian Academy of Sciences} in Linz in 2024.

\appendix
\section{Equations for \texorpdfstring{$\OCM$}{the oriented Cayley-Menger variety}}
\label{equations}

\begin{lemma}
\label{lemma:equation_oriented_volume}
 The following equations hold for $\OCM$. Form the matrix
  \[
  M_{i, i+1, \Top, \Bottom} =
  \left(
  \begin{array}{rrrrr}
    0 & 1 & 1 & 1 & 1 \\
    1 & 0 & d_{i \, i+1} & d_{i \, \Top} & d_{i \, \Bottom} \\
    1 & d_{i \, i+1} & 0 & d_{i+1 \, \Top} & d_{i+1 \, \Bottom} \\
    1 & d_{i \, \Top} & d_{i+1 \, \Top} & 0 & d_{\Top \, \Bottom} \\
    1 & d_{i \, \Bottom} & d_{i+1 \, \Bottom} & d_{\Top \, \Bottom} & 0 \\
  \end{array}
  \right)\,.
  \]
 Then
 \[
   288 \, s_{i,i+1}^2 - \det(M_{i, i+1, \Top, \Bottom}) = 0 \,.
 \]
\end{lemma}
\begin{proof}
 This is a well known fact in distance geometry (see for example \cite[Section~4.1.2]{Liberti2017}).
\end{proof}

The following lemma has content very similar to \cite{Sabitov};
we state and prove it because for our purposes we need a precise description of the equation satisfied by the distances and the oriented volumes.

\begin{lemma}
\label{lemma:equation_distance_volume}
 The following equations hold for $\OCM$:
 \[
  a_k d_{ij} + b_{ijk} - 288 \varsigma_{ik} \varsigma_{kj} = 0 \,,
 \]
 where
 \[
  a_k = 2(d_{k\Top} d_{k\Bottom} + d_{k\Top} d_{\Top\Bottom} + d_{k\Bottom} d_{\Top\Bottom}) - (d_{k\Top}^2 + d_{k\Bottom}^2 + d_{\Top\Bottom}^2)
 \]
 \begin{align*}
  b_{ijk} =
   &\bigl(d_{ik} (d_{j\Bottom}-d_{j\Top})+d_{kj}
   (d_{i\Bottom}-d_{i\Top})+d_{k\Bottom}
   (d_{i\Top}+d_{j\Top})-d_{k\Top} (d_{i\Bottom}+d_{j\Bottom})\bigr)
   (d_{k\Bottom}-d_{k\Top})\\
   &-d_{\Top\Bottom} \bigl(d_{i\Top}
   (d_{k\Bottom}+d_{kj})+d_{i\Bottom}
   (d_{k\Top}+d_{kj})+d_{j\Top}
   (d_{ik}+d_{k\Bottom})+d_{j\Bottom} (d_{ik}+d_{k\Top})\bigr)\\
   &-(d_{k\Bottom}+d_{k\Top}-d_{\Top\Bottom})
   \bigl(d_{\Top\Bottom} (d_{ik}+d_{kj})+(d_{i\Top}
   d_{j\Bottom}+d_{j\Top} d_{i\Bottom})\bigr)\\
   &+2 (d_{i\Top} d_{k\Bottom}
   d_{j\Top}+d_{i\Bottom} d_{k\Top} d_{j\Bottom}+d_{k\Bottom}
   d_{k\Top} d_{\Top\Bottom} + d_{ik} d_{kj} d_{\Top\Bottom})\,.
 \end{align*}
 Notice that both $a_{k}$ and $b_{ijk}$ only depend on the $d$-coordinates; in particular, $a_k$ depends on the variables $d_{k\Top}$,$d_{k\Bottom}$, and $d_{\Top\Bottom}$, while $b_{ijk}$ depends on the variables $d_{m\Top}$, $d_{m\Bottom}$ for $m \in \{i,j,k\}$ and $d_{ik}$, $d_{kj}$, and $d_{\Top\Bottom}$.
\end{lemma}
\begin{proof}
 It is well-known (see for example \cite[Theorem 4.2]{Liberti2016}) that the following determinant is identically zero:
 \[
  \det
  \left(
  \begin{array}{cccccc}
   0 & 1 & 1 & 1 & 1 & 1 \\
   1 & 0 & d_{ik} & d_{ij} & d_{i\Top} & d_{i\Bottom} \\
   1 & d_{ik} & 0 & d_{kj} & d_{k\Top} & d_{k\Bottom} \\
   1 & d_{ij} & d_{kj} & 0 & d_{j\Top} & d_{j\Bottom} \\
   1 & d_{i\Top} & d_{k\Top} & d_{j\Top} & 0 & d_{\Top\Bottom} \\
   1 & d_{i\Bottom} & d_{k\Bottom} & d_{j\Bottom} & d_{\Top\Bottom} & 0 \\
  \end{array}
  \right) \,.
 \]
 This determinant is a polynomial in $d_{ij}$ of degree $k$.
 Solving the equation with respect to $d_{ij}$ yields the desired equation,
 with an ambiguity regarding the sign of
 the monomial~$288 \varsigma_{i,k} \varsigma_{k,j}$.
 To solve the ambiguity, one can plug in the parametrization of $\OCM$
 with symbolic coordinates for the points in $(\R^3)^{V_{\Bipyramid}}$.
 See notebook \cite[Section~3.1]{supportingCode}.
\end{proof}

\begin{definition}
 We denote by $W_{i, j, k, \ell}$ the squared volume of the tetrahedron $(i, j, k, \ell)$.
 It follows that $W_{i, j, k, \ell}$ is a polynomial in $\{ d_{uv} \}_{u,v \in \{i, j, k, \ell\}}$.
\end{definition}

\begin{lemma}
\label{lemma:oriented_volume_diagonal}
 For every $i, j, k \in \{1, \dotsc, n\}$, we have
 \[
   a_{i,j,k} \, \varsigma_{i, j} + b_{i,j,k} = 0 \,,
 \]
 where
 \begin{align*}
  a_{i,j,k} &= -4 (\varsigma_{i, k} + \varsigma_{k, j}) \bigl((\varsigma_{i, k} + \varsigma_{k, j})^2 - W_{i, k, j, \Bottom} - W_{i, k, j, \Top} + W_{i, j, \Top, \Bottom} \bigr) \, ,  \\
  b_{i,j,k} &= (\varsigma_{i, k} + \varsigma_{k, j})^2 \bigl( (\varsigma_{i, k} + \varsigma_{k, j})^2 - 2W_{i, k, j, \Bottom} - 2W_{i, k, j, \Top} + 6W_{i, j, \Top, \Bottom} \bigr) \\
  &\phantom{=} +
  (W_{i, k, j, \Bottom} + W_{i, k, j, \Top} - W_{i, j, \Top, \Bottom})^2 - 4W_{i, k, j, \Bottom} W_{i, k, j, \Top} \, .
 \end{align*}
\end{lemma}
\begin{proof}
 See notebook~\cite[Section~3.2]{supportingCode}.
\end{proof}

\begin{corollary}
\label{corollary:rational_functions}
 For every $i, j \in \{1, \dotsc, n\}$, we have that both $d_{ij}$ and $\varsigma_{i,j}$
 are rational functions in $\{s_e\}_{e \in \Eq} \cup \{d_{e}\}_{e \in E_{\Bipyramid}} \cup \{d_{\Top\Bottom}\}$ in the function field of $\OCM$.
\end{corollary}
\begin{proof}
 We prove the statement by induction on $|j - i|$.
 Since $d_{ji} = d_{ij}$ and $\varsigma_{i,j} = -\varsigma_{j,i}$, we can suppose that $j > i$.
 For the base case $j = i+1$, then $\varsigma_{i,i+1} = s_{i,i+1}$, so the statement is proven.
 By choosing $k = j-1$, \Cref{lemma:equation_distance_volume} shows that $d_{ij}$ is a rational function in
 $\{\varsigma_{i, j-1}, \varsigma_{j-1, j}\} \cup \{ d_{i, j-1}, d_{j-1, j}\} \cup \{d_{u\Top}, d_{u\Bottom} \}_{u \in \{i, j-1, j\}} \cup \{d_{\Top\Bottom}\}$.
 Moreover, again by choosing $k = j-1$, \Cref{lemma:oriented_volume_diagonal} shows that $\varsigma_{i,j}$ is a rational function in $\{\varsigma_{i, j-1}, \varsigma_{j-1, j}\} \cup \{ d_{i, j-1}, d_{j-1, j}, d_{i, j} \} \cup \{d_{u\Top}, d_{u\Bottom} \}_{u \in \{i, j-1, j\}} \cup \{d_{\Top\Bottom}\}$: indeed, the squared volumes appearing in $a_{i,j,k}$ and $b_{i,j,k}$ can be expressed as polynomials in $\{ d_{i, j-1}, d_{j-1, j}, d_{i, j} \} \cup \{d_{u\Top}, d_{u\Bottom} \}_{u \in \{i, j-1, j\}} \cup \{d_{\Top\Bottom}\}$.
 Hence, by the inductive step, both $d_{ij}$ and $\varsigma_{i,j}$ are rational functions in the desired variables.
\end{proof}

\phantomsection
\addcontentsline{toc}{section}{References}
\bibliographystyle{alphaurl}
\bibliography{galois}

\bigskip
\bigskip

\textsc{(MG) University of Trieste,
Department of Mathematics, Informatics and Geosciences,
Via Valerio 12/1, 34127 Trieste, Italy}\\
Email address: \texttt{matteo.gallet@units.it}

\textsc{(GG) Johann Radon Institute for Computation and Applied Mathematics (RICAM), Austrian
Academy of Sciences, Linz, Austria}\\
Email address: \texttt{georg.grasegger@ricam.oeaw.ac.at}

\textsc{(JL) Department of Applied Mathematics, Faculty of Information Technology, Czech Technical University in Prague, Prague, Czech Republic}\\
Email address: \texttt{jan.legersky@fit.cvut.cz}

\textsc{(JS) Johannes Kepler University Linz, Research Institute for Symbolic Computation (RISC), Linz, Austria}\\
Email address: \texttt{jschicho@risc.jku.at}

\end{document}